\documentclass[a4paper]{amsart}
\usepackage[english]{babel}
\usepackage[utf8x]{inputenc}
\usepackage[T1]{fontenc}
\usepackage{amsthm}
\usepackage{dsfont}
\usepackage{amssymb}
\usepackage{amsmath}
\theoremstyle{plain}
\usepackage{chngcntr}
\usepackage{relsize}
\usepackage{pdfpages}

\newtheorem{Lemma}{Lemma}
\newtheorem{Theorem}[Lemma]{Theorem}
\newtheorem{Hypothesis}[Lemma]{Hypothesis}
\newtheorem{Proposition}[Lemma]{Proposition}
\newtheorem{Corollary}[Lemma]{Corollary}

\usepackage[a4paper,top=3cm,bottom=2cm,left=3cm,right=3cm,marginparwidth=1.75cm]{geometry}

\usepackage{float}
\usepackage{nccmath}
\usepackage{mathtools}
\usepackage{amsfonts}
\usepackage{amsmath}
\usepackage{graphicx}
\usepackage[colorinlistoftodos]{todonotes}
\usepackage[colorlinks=true, allcolors=blue]{hyperref}
\makeatletter
\newcommand*{\rom}[1]{\expandafter\@slowromancap\romannumeral #1@}
\makeatother

\title{The large sieve for square moduli, revisited}

\subjclass[2010]{11L07,11L40,11J25,11N35}

\keywords{large sieve, square moduli, additive energy, modular square roots, exponential sums}

\author{Stephan Baier}
\address{Stephan Baier,
Ramakrishna Mission Vivekananda Educational and Research Institute, Department of Mathematics, G. T. Road, PO Belur Math, Howrah, West Bengal 711202, India}
\email{stephanbaier2017@gmail.com}

\begin{document}
\maketitle
\begin{abstract} We revisit the large sieve for square moduli and obtain conditional improvements under hypotheses on higher additive energies of modular square roots.  
\end{abstract}

\tableofcontents

\section{Current state of the art}
\subsection{Review of the classical large sieve}
Throughout this article, following usual custom, we assume that $\varepsilon$ is an arbitrarily small but fixed positive number. For a real number $x$ and a positive integer $r$, we set 
$$
e(x):=e^{2\pi i x} \quad \mbox{and} \quad e_r(x):=e\left(\frac{x}{r}\right). 
$$

Large sieve inequalities have turned out extremely useful in analytic number theory. The classical example is the large sieve inequality for additive characters, which can be easily converted into a large sieve inequality for multiplicative characters. It originated in Linnik's work \cite{Lin} on the least quadratic non-residue modulo primes. This inequality is a very flexible tool: It can be turned into an upper bound sieve but also be used to understand the distribution of arithmetic functions in residue classes on average, for example. Its statement is as follows (see \cite[Satz 5.2.2]{Bru} for the specific case of Farey fractions). 

\begin{Theorem} \label{classicallargesieve}
Let $Q,N\in \mathbb{N}$, $M\in \mathbb{Z}$ and $(a_n)_{M<n\le M+N}$ be any sequence of complex numbers.  Then
\begin{equation} \label{lsclassical}
\sum\limits_{q\le Q} \sum\limits_{\substack{a=1\\ (a,q)=1}}^q \left| \sum\limits_{M<n\le M+N} a_ne\left(\frac{na}{q}\right)\right|^2\le (Q^2+N-1)\sum\limits_{M<n\le M+N} |a_n|^2. 
\end{equation}
\end{Theorem}

The above inequality is sharp. Yet, there are many unresolved questions about variations of this inequality.

\subsection{Known results on the large sieve with sparse sets of moduli} \label{knownresults}
One way to modify the left-hand side of \eqref{lsclassical} is to restrict the moduli $q$ to a sparse subset $\mathcal{S}$ of the integers. Now the goal is to establish an inequality of the form
\begin{equation} \label{lssparse}
\sum\limits_{\substack{q\in \mathcal{S}\\ q\le Q}} \sum\limits_{\substack{a=1\\ (a,q)=1}}^q \left| \sum\limits_{M<n\le M+N} a_ne\left(\frac{na}{q}\right)\right|^2\le \Delta_{\mathcal{S}}(Q,N)Z,
\end{equation} 
where $\Delta_{\mathcal{S}}(Q,N)$ is a function of an as small as possible magnitude. Here and in the following, we set
$$
Z:=\sum\limits_{M<n\le M+N} |a_n|^2.
$$
Trivially, using \eqref{lsclassical}, the bound \eqref{lssparse} holds for $\Delta(Q,N)=Q^2+N-1$, but we aim to do better. By heuristic considerations one may expect that for generic sets $\mathcal{S}$ one may take $\Delta_{\mathcal{S}}(Q,N)$ roughly to be the number of possible fractions $a/q$ plus the length $N$ of the trigonometrical polynomial, which is about of size $Q\sharp\{q\in \mathcal{S}: q\le Q\}+N$. However, this may not be true in general as it is possible that there are points around which the fractions $a/q$ accumulate (see \cite{BLZ} for the case of square moduli). In section \ref{WolApp}, we will investigate the connection between $\Delta_{\mathcal{S}}(Q,N)$ and the distribution of fractions $a/q$ with $q\in \mathcal{S}$ and $(a,q)=1$ in more detail.

To my knowledge, the first who considered a situation as described above was Wolke who studied the case of prime moduli in \cite{Wol}. He proved that if $\mathcal{S}$ is the set of primes, then we may take
$$
\Delta_{\mathcal{S}}(Q,N)=\frac{Q^2\log \log Q}{\log(Q/\sqrt{N})},
$$
provided that $Q> \sqrt{N}$. The factor $\log\log Q$ is slightly unsatisfactory and was removed by Iwaniec \cite{Iwa} under the extra condition that the summation variable $n$ is free of a sufficient number of small prime divisors. 

Another sparse set $\mathcal{S}$ was later considered by Zhao who studied the case of square moduli in \cite{Zhao1}. There have been many subsequent works on square, and, more generally, power moduli, providing improvements as well as different approaches (see \cite{Bai1}, \cite{Bai2}, \cite{BaiLyn}, \cite{BaiZhao2}, \cite{BaiZhao1}, \cite{BMS}, \cite{Hal1}, \cite{Hal2}, \cite{Hal3}, \cite{Hal}). These papers use a variety of tools such as Diophantine approximation, elementary counting arguments, Fourier analysis, bounds arising from Vinogradov's mean value theorem, additive energy of sequences. In this paper, we revisit the case of square moduli and make conditional progress.       

\subsection{Known results on the large sieve for square moduli} \label{known}
For $\mathcal{S}$ the set of squares of integers, using Fourier analytic tools, Zhao established in \cite{Zhao1} that \eqref{lssparse} holds with 
$$
\Delta_{\mathcal{S}}(Q_0,N)=(Q_0N)^{\varepsilon}\left(Q_0^{3/2}+Q_0N^{1/2}+Q_0^{1/4}N\right)
$$
for all positive integers $Q_0$ and $N$. In other words, setting $Q=Q_0^{1/2}$, we have that
\begin{equation} \label{Zhaobound}
\sum\limits_{q\le Q}\sum\limits_{\substack{a=1\\ (a,q)=1}}^{q^2} \left|\sum\limits_{M<n\le M+N} a_ne\left(\frac{na}{q^2}\right)\right|^2\ll \tilde{\Delta}(Q,N)Z
\end{equation}
with 
\begin{equation} \label{tildeDelta}
\tilde{\Delta}(Q,N)=(QN)^{\varepsilon}\left(Q^3+Q^2N^{1/2}+Q^{1/2}N\right).
\end{equation}
Zhao also conjectured that we may take
\begin{equation} \label{conj}
\tilde{\Delta}(Q,N)=(QN)^{\varepsilon}(Q^3+N).
\end{equation}
Note that $Q^3$ is roughly the number of Faray fractions $a/q$ on the left-hand side of \eqref{Zhaobound} and $N$ is the length of the trigonometrial polynomial. On the other side,  the author of the present paper, Lynch and Zhao proved in \cite{BLZ} that \eqref{Zhaobound} does not hold with $\tilde{\Delta}(Q,N)=Q^3+N$. 
Based on Wolke's method from \cite{Wol}, the author of the present paper showed in \cite{Bai1} using Diophantine approximation and elementary counting arguments that the term $Q^{1/2}N$ in \eqref{tildeDelta} can be replaced by $N$. This gave an improvement for $Q\le N^{1/3-\varepsilon}$. In \cite{BaiZhao1}, merging their methods, Zhao and the author of the present paper subsequently replaced \eqref{tildeDelta} by 
\begin{equation} \label{tildeDelta'}
\tilde{\Delta}(Q,N)=(QN)^{\varepsilon}\left(Q^3+N+\min\left\{Q^2N^{1/2},Q^{1/2}N\right\}\right),
\end{equation}
which gave another improvement for $Q\ge N^{1/3+\varepsilon}$. However,  since Zhao's initial result in \cite{Zhao1} there has been no progress at the point $Q=N^{1/3}$ for over twenty years. This point may be viewed as the critical point, where both terms $Q^3$ and $N$ in the conjectural bound stated in \eqref{conj} coincide. For $Q=N^{1/3}$, we still have no more than a bound of $\tilde{\Delta}(Q,N)=Q^{7/2+\varepsilon}$, which is a factor of $Q^{1/2}$ off the conjectured bound $\tilde{\Delta}(Q,N)=Q^{3+\varepsilon}$. It would be highly desirable to lower this extra factor $Q^{1/2}$ to $Q^{1/2-\eta}$ for some $\eta>0$. However, the exponent $1/2$ seems to be a barrier - so far all attempts to break it failed. In this paper, we break it conditionally under natural hypotheses on higher additive energies of modular square roots. Throughout the sequel, we will assume that 
\begin{equation} \label{QNcond}
Q^2\le N\le Q^4
\end{equation}
since if $N$ lies outside this range, then \eqref{tildeDelta'} implies Zhao's conjectural bound. Let us record the best known bound for the large sieve with square moduli in \cite{BaiZhao1} as a theorem below. 

\begin{Theorem} \label{bestknownls} Let $Q,N\in \mathbb{N}$, $M\in \mathbb{Z}$ and $(a_n)_{M<n\le M+N}$ be any sequence of complex numbers.  Then
\begin{equation*}
\sum\limits_{q\le Q}\sum\limits_{\substack{a=1\\ (a,q)=1}}^{q^2} \left|\sum\limits_{M<n\le M+N} a_ne\left(\frac{na}{q^2}\right)\right|^2\ll (QN)^{\varepsilon}\left(Q^3+N+\min\left\{Q^2N^{1/2},Q^{1/2}N\right\}\right)Z.
\end{equation*}
\end{Theorem} 

\subsection{Applications} The large sieve for square moduli has found a number of applications. In \cite{BaiZhao}, the authors of \cite{BaiZhao1} used it to establish the infinitude of primes of the form $aq^2+1$ with $a\le q^{5/9+\varepsilon}$. Matom\"aki \cite{Mat} improved the exponent $5/9$ to $1/2$ by injecting Harman's sieve in the method. Further progress was made by Merikoski \cite{Mer} who replaced the exponent $1/2$, which can be viewed as a natural barrier, by $1/2-\eta$ for some small $\eta>0$. The large sieve for square moduli also played a role in the paper \cite{BFKS} by Bourgain, Ford, Konyagin and Shparlinski which investigated divisibility of Fermat quotients. It was also applied in the papers \cite{BPS} by Banks, Pappalardi and Shparlinski and \cite{SZ} by Shparlinski and Zhao which investigated quantitative properties of families of elliptic curves over finite fields.    

\section{Additive energies of modular square roots and main result}
Additive energy is an important concept in additive combinatorics. During the recent years, it has turned out to be very useful in analytic number theory too (see \cite{BMS}, \cite{DuZa}, \cite{DKSZ}, \cite{May}, \cite{KSSZ}, \cite{SSZ22}, \cite{SSZ},  for example) and will likely continue to be. This includes the work \cite{BMS} by Baker, Munsch and Shparlinski on the large sieve with sparse sets of moduli. Our article builds on recent investigations of additive energies related to modular square roots. This development was started off in the paper \cite{DuZa} by Dunn and Zaharescu (published first on the arXiv preprint server in 2019 (arXiv:1903.03416) and appeared in 2025), where modular square roots arose in connection with bilinear forms of Sali\'e sums. This line of investigations was continued in the paper \cite{DKSZ} by Dunn, Kerr, Shparlinski and Zaharescu, where improvements were obtained. Our article relies on subsequent results by Kerr, Shkredov, Shparlinski and Zaharescu in \cite{KSSZ} and \cite{SSZ} (see also \cite{SSZ22} for an average result with a power saving). A concise summary of this sequence of works can be found in \cite[section 1, after Theorem 1.2]{DuZa}. In this section, we describe the results on additive energies of modular square roots in \cite{KSSZ} and \cite{SSZ}, state hypotheses and formulate our main result of this article. 
 
Throughout this article, by abuse of notation, we denote a square root of an integer $m$ modulo a natural number $r$ by $\sqrt{m} \bmod r$, if it exists. So $\sqrt{m}$ stands for the {\it collection of all} $k \bmod q$ such that 
$k^2\equiv m \bmod{r}$. If $R\ge 1$ and $(j,r)=1$, we denote by $E_2(R;j,r)$ the additive energy of the set of modular square roots $\sqrt{jm} \bmod r$, where $m$ runs over all natural numbers not exceeding $R$. So we set 
\begin{equation} \label{AE}
E_2(R;j,r):=\sum\limits_{\substack{1\le m_1,m_2,m_3,m_4\le R\\ \sqrt{jm_1}+\sqrt{jm_2}\equiv \sqrt{jm_3}+\sqrt{jm_4}\bmod r}} 1 = \sum\limits_{\substack{(k_1,k_2,k_3,k_4)\bmod r\\ k_1+k_2\equiv k_3+k_4\bmod{r}\\ k_i^2\equiv jm_i \text{ for some } m_i\in \mathbb{N}\cap [1,R]}} 1.
\end{equation}
We also define the higher additive energy $E_4(R;j,r)$ by
\begin{equation} \label{AE4}
E_4(R;j,r):=\sum\limits_{\substack{1\le m_1,...,m_8\le R\\ \sqrt{jm_1}+\cdots + \sqrt{jm_4}\equiv \sqrt{jm_5}+\cdots +\sqrt{jm_8}\bmod r}} 1 = \sum\limits_{\substack{(k_1,...,k_8)\bmod r\\ k_1+\cdots +k_4\equiv k_5+\cdots +k_8\bmod{r}\\ k_i^2\equiv jm_i \text{ for some } m_i\in \mathbb{N}\cap [1,R]}} 1.
\end{equation}
For {\it prime} moduli $r$, Kerr, Shkredov, Shparlinski and Zaharescu \cite{KSSZ} recently established the following bound for the additive energy for modular square roots.

\begin{Theorem}[Theorem 1.1. in \cite{KSSZ}] \label{E2theo}
Let $r$ be a prime. Then for every $j\in \mathbb{Z}$ coprime to $r$ and any integer $R\le r$, we have 
\begin{equation} \label{E2result}
E_2(R;j,r)\ll \left(\frac{R^{3/2}}{r^{1/2}}+1\right)R^{2+\varepsilon}. 
\end{equation}
\end{Theorem}

As remarked in \cite{KSSZ}, the higher additive energy $E_4(R;j,r)$ satisfies a bound of 
$$
E_4(R;j.r)\ll R^4E_2(R;j,r),
$$
and hence Theorem \ref{E2theo} implies the bound 
\begin{equation} \label{establishedbound}
E_4(R;j,r)\ll \left(\frac{R^{3/2}}{r^{1/2}}+1\right)R^{6+\varepsilon}. 
\end{equation}
In the same paper \cite{KSSZ}, the authors improved this bound for $R\le r^{1/12}$, establishing the following. 

\begin{Theorem}[Theorem 1.2. in \cite{KSSZ}] \label{E4theo}
Let $r$ be a prime. Then for every $j\in \mathbb{Z}$ coprime to $r$ and any integer $R\le r$, we have 
\begin{equation} \label{E4result}
E_4(R;j,r)\ll \left(\frac{R^{5/8}}{r^{1/8}}+\frac{R^{11/2}}{r^{1/2}}+\frac{R^3}{r^{1/4}}\right)R^{6+\varepsilon}+R^{5+\varepsilon}.
\end{equation}
\end{Theorem}

In \cite{SSZ}, Shkredov, Shparlinski and Zaharescu showed that for almost all primes $r$ and all $R<r$ and $j$ coprime to $r$, the expected bound
\begin{equation} \label{2expect}
E_2(R;j,r)\ll \left(\frac{R^{4}}{r}+R^2\right)r^{\varepsilon}
\end{equation}
holds. This may be conjectured to hold for {\it all} primes $r$: Heuristically, the probability of a random 4-tuple $(m_1,m_2,m_3,m_4)\in \{1,...,R\}^4$ to satisfy the congruence 
$$
\sqrt{jm_1}+\sqrt{jm_2}\equiv \sqrt{jm_3}+\sqrt{jm_4}\bmod r
$$
for a suitable choice of modular square roots of $jm_i$ (if existent) is $1/r$. This yields the term $R^4/r$ in \eqref{2expect}. Moreover, we have a diagonal contribution coming from the 4-tuples of the form $(m_1,m_2,m_1,m_2)$, which yields the term $R^2$ in \eqref{2expect}. Theorem \ref{E2theo} establishes this for $R\le r^{1/3}$. Similarly, we may expect a bound of 
\begin{equation} \label{4expect}
E_4(R;j,r)\ll \left(\frac{R^{8}}{r}+R^4\right)r^{\varepsilon}
\end{equation}
for the four-fold additive energy, where the term $R^8/r$ arises from the probability $1/r$ of a random 8-tuples $(m_1,...,m_8)$ to satisfy the congruence
$$
\sqrt{jm_1}+\cdots +\sqrt{jm_4}\equiv \sqrt{jm_5}+\cdots +\sqrt{jm_8}\bmod r
$$
and the term $R^4$ accounts for the diagonal contribution of 8-tuples of the form $(m_1,...,m_4,m_1,...,m_4)$. 

It is conceivable that the methods used to prove the above Theorems \ref{E2theo} and \ref{E4theo} in \cite{KSSZ} can be made work for all moduli $r$, giving results similar to \eqref{E2result} and \eqref{E4result} without assuming $r$ to be a prime. Care needs to be taken, though: Whereas for prime moduli $r$, the number of modular square roots of a given $m \bmod r$ is bounded by 2, this number may be considerably larger if $r$ is not a prime. For example, if $r=p^2$ is the square of a prime, then $m=0$ has precisely $p$ square roots modulo $r$, namely $\sqrt{m}=kp$ with $k=0,...,p-1$. More generally, let $r=p^{\alpha}$ be a prime power and $0\le m=m_1p^{\beta}<r$ with $(m_1,p)=1$. If $m\equiv k^2\bmod{r}$, then $k=k_1p^{\lceil \beta/2 \rceil}$ for some $k_1\bmod{p^{\alpha-\lceil \beta/2\rceil}}$, and the congruence reduces to $m_1\equiv k_1^2p^{2\lceil \beta/2\rceil-\beta} \bmod{p^{\alpha-\beta}}$. If $m\not=0$, then $\beta<\alpha$ and $\beta$ is necessarily even. It follows that $m_1\equiv k_1^2\bmod{p^{\alpha-\beta}}$. Using Hensel's lemma, the number of solutions $k_1\bmod{p^{\alpha-\beta/2}}$ of this congruence can be as large as $2p^{\beta/2}$ if $p$ is odd and $2^{\beta/2+2}$ if $p=2$. In the case when $m=0$, there are exactly $p^{\lfloor \alpha/2\rfloor}$ square roots of $m$ modulo $r$. Now using the Chinese remainder theorem, for general moduli $r$, the number of modular square roots of any $m\bmod{r}$ is bounded by $O_{\varepsilon}\left(r^{\varepsilon}(m,r)^{1/2}\right)$. It follows that the number of modular square roots modulo $r$ of all $m$ in the range $1\le m\le R$ is bounded by
$$
\ll r^{\varepsilon}\sum\limits_{1\le m\le R} (m,r)^{1/2}\ll Rr^{2\varepsilon},
$$  
where we use Lemma \ref{gcdsums} below. This is only slightly larger than $R$. So even if $r$ is not a prime, we may still expect the bounds \eqref{2expect} and \eqref{4expect} to hold. Thus, we propose the following hypotheses. 

\begin{Hypothesis} \label{H1}  Let $r$ be a natural number. Then the bound \eqref{2expect} holds for every $j\in \mathbb{Z}$ coprime to $r$ and any integer $R\le r$.
\end{Hypothesis}

\begin{Hypothesis}  \label{H2} Let $r$ be a natural number. Then the bound \eqref{4expect} holds for every $j\in \mathbb{Z}$ coprime to $r$ and any integer $R\le r$.
\end{Hypothesis}

Here it is essential that we have restricted the $m_i$'s to the interval $1\le m_i\le R$, i.e., we have excluded the possibility that $m_i=0$. If we would include $m_i=0$, then the above hypotheses were not true - we would get extra terms of sizes $r^2$ and $r^4$, respectively.    

We note that whereas Theorem \ref{E2theo} establishes Hypothesis \ref{H1} for an amazingly large range $R\le r^{1/3}$ in the case of prime moduli $r$, Theorem \ref{E4theo} is quite far away from Hypothesis \ref{H2}. However, these results in \cite{KSSZ} are very recent, so there may be space for improvements. 

Since it will be crucial to use Weyl differencing in this paper, we will also face additive energies of differences of modular square roots of the form
$$
f_{j,h}(m):=\sqrt{j(m+h)}-\sqrt{jm} \bmod{r},
$$    
if existent, where $h$ is a non-zero integer. Again, by abuse of notation, $f_{j,h}(m)$ really denotes the {\it collection} of differences between all possible modular square roots of $j(m+h)$ and $jm$. So we define 
\begin{equation} \label{F2def}
F_2(R;j,h,r):=\sum\limits_{\substack{1\le m_1,m_2,m_3,m_4\le R\\ f_{j,h}(m_1)+f_{j,h}(m_2)\equiv f_{j,h}(m_3)+f_{j,h}(m_4)\bmod r}} 1 = \sum\limits_{\substack{(k_1,\tilde{k}_1,k_2,\tilde{k}_2,k_3,\tilde{k}_3,k_4,\tilde{k}_4)\bmod r\\ (\tilde{k}_1-k_1)+(\tilde{k}_2-k_2)\equiv (\tilde{k}_3-k_3)+(\tilde{k}_4-k_4)\bmod{r}\\ \tilde{k}_i^2\equiv j(m_i+h) \text{ and } k_i^2\equiv jm_i \text{ for some } m_i\in \mathbb{N}\cap [1,R]}} 1.
\end{equation}
On the basis of similar heuristic arguments as above, we conjecture the following. 

 \begin{Hypothesis} \label{H3}  Let $r$ be a natural number. Then the bound 
\begin{equation} \label{3expect}
F_2(R;j,h,r)\ll \left((h,r)\cdot \frac{R^{4}}{r}+R^2\right)r^{\varepsilon}
\end{equation}
holds for every $j\in \mathbb{Z}$ coprime to $r$ and any integer $R\le r$.
\end{Hypothesis}

In the following, we indicate why we include an extra factor of $(h,r)$ in \eqref{3expect}. If $r$ is odd, then by the considerations at the beginning of the appendix (section 14), for any given $b \bmod{r}$ with $(b,r)=d$, the system of congruences  
\begin{equation*}
\begin{cases}
b\equiv \tilde{k}-k\bmod{r}\\
k^2\equiv ja\bmod{r}\\ \tilde{k}^2 \equiv j(a+h)\bmod{r}
\end{cases}
\end{equation*}
is solvable for $k,\tilde{k},a$ if and only if $d|h$. In this case, by the said considerations in the appendix, making a change of variables $c=\tilde{k}+k$, this system reduces to a single congruence of the form
$$
c\equiv \overline{b_1}jh_1\bmod{r_1} \quad \mbox{with } b_1=\frac{b}{d}, h_1=\frac{h}{d} \mbox{ and } r_1=\frac{r}{d},
$$
and we therefore get exactly $d=r/r_1$ solutions $(m_1,m_2,a) \bmod{r}$. In particular, if $(b,r)=d=(h,r)$, then the modulus $r$ reduces to $r_1=r/(h,r)$, which is the least possible modulus of the above congruence. We therefore need to replace the quotient $R^4/r$ in our heuristic by $(h,r)R^4/r$ because the counting problem is really to be taken modulo $r/(h,r)$. In particular, if $h=0$, we get a term of size $R^4$, which is the correct order of magnitude of $F_2(R;j,h,r)$ in this case because $f_{j,h}(m)\equiv 0\bmod{r}$ for all $m \bmod{r}$. If $r$ is even, then the arguments are similar. 

To the author's knowledge, there is no unconditional result on the additive energy $F_2(R;j,h,r)$ in the literature. 
Our main result in this article is the following.

\begin{Theorem} \label{mainresult} Let $Q\in \mathbb{N}$, $M\in \mathbb{Z}$, $N=Q^3$ and $a_n$ a sequence of complex numbers. Suppose that Hypotheses \ref{H2} and \ref{H3} hold. Then
\begin{equation*}
\sum\limits_{q\le Q}\sum\limits_{\substack{a=1\\ (a,q)=1}}^{q^2} \left|\sum\limits_{M<n\le M+N} a_ne\left(\frac{na}{q^2}\right)\right|^2\ll Q^{7/2-1/135}\sum\limits_{M<n\le M+N} |a_n|^2.
\end{equation*}
\end{Theorem}

This saves a factor of $Q^{1/135}$ over all known bounds if $N=Q^3$. Our method also gives a conditional improvement of the best known bound \eqref{tildeDelta'} if $\log_Q N$ is in some neighborhood of $3$, but we have not worked out the details. One can infer such a result from the later Proposition \ref{Propmain}, but the calculations are complicated.    \\ \\
{\bf Acknowledgements.} The author wishes to thank the anonymous referee for carefully going through this article and the Ramakrishna Mission Vivekananda Educational and Research Institute for providing excellent working conditions.  
  
\section{Preliminaries from the theory of exponential sums}
The following preliminaries are used in this article.

To reduce the sizes of amplitude functions, we will employ the following version of Weyl differencing for weighted exponential sums.

\begin{Proposition}[Weyl differencing] \label{Weyldif}
Let $I=(a,b]$ be an interval of length $|I|=b-a\ge 1$. Let $f: I \rightarrow \mathbb{R}$ and $\Phi:\mathbb{R}\rightarrow \mathbb{C}$ be functions, where $|\Phi(x)|\le 1$ for all $x\in \mathbb{R}$. Then for every $H\in \mathbb{R}$ with $1\le H\le |I|$, we have the bound
$$
\left|\sum\limits_{n\in I} \Phi(n) e(f(n))\right|^2\ll \frac{|I|^2}{H}+\frac{|I|}{H}\cdot \sum\limits_{1\le h\le H} \left|\sum\limits_{n\in I_h} \Phi(n) \Phi(n+h) e(f(n+h)-f(n))\right|,
$$ 
where $I_h:=(a,b-h]$. 
\end{Proposition}

\begin{proof}
This is a consequence of \cite[Lemma 2.5.]{GrKo}.
\end{proof} 

For a convenient application of the Poisson summation formula, we will use smooth weight functions $\Phi:\mathbb{R} \rightarrow \mathbb{C}$ which fall into the Schwartz class (for details on the Schwartz class, see \cite{StSh}). For a Schwartz class function as above, we define its Fourier transform $\hat{\Phi}:\mathbb{R} \rightarrow \mathbb{C}$ as 
$$
\hat{\Phi}(y):=\int\limits_{\mathbb{R}} \Phi(x)e(-xy){\rm d}x.
$$ 
Below is a version of the Poisson summation formula in which the sum on the left-hand side runs over a residue class.

\begin{Proposition}[Poisson summation] \label{Poisum} Let $\Phi:\mathbb{R}\rightarrow \mathbb{C}$ be a Schwartz class function. Let $L\ge 1$, $M\in \mathbb{R}$, $r\in \mathbb{N}$ and $a\in \mathbb{Z}$. Then
$$
\sum\limits_{n\equiv a\bmod{r}} \Phi\left(\frac{n-M}{L}\right)=\frac{L}{r} \sum\limits_{n\in \mathbb{Z}} \hat{\Phi}\left(\frac{nL}{r}\right) e_r\left(n(a-M)\right). 
$$
\end{Proposition}

\begin{proof} This arises by a linear change of variables from the well-known basic version of the Poisson summation formula which states that
$$
\sum\limits_{n\in \mathbb{Z}} F(n)=\sum\limits_{n\in \mathbb{Z}} \hat{F}(n)
$$
for any Schwartz class function $F:\mathbb{R}\rightarrow \mathbb{C}$ (see \cite{StSh}).
\end{proof}

In certain cases, we will estimate exponential sums using the following result.

\begin{Proposition} \label{analyticexpsumbound} Let $I=(a,b]$ be an interval of length $|I|=b-a\ge 1$ and $k$ be a positive integer. Suppose that $f: I \rightarrow \mathbb{R}$ is a function with $k+2$ continuous derivatives on $I$. Suppose also that there exist $\lambda>0$ and $\alpha\ge 1$  satisfying
$$
\lambda\le |f^{(k+2)}(x)|\le \alpha\lambda 
$$
for all $x\in I$. Set $K:=2^k$. Then we have the bound
$$
\sum\limits_{n\in I} e(f(n))\ll |I|(\alpha^2\lambda)^{1/(4K-2)}+|I|^{1-1/(2K)}\alpha^{1/(2K)}+|I|^{1-2/K+1/K^2}\lambda^{-1/(2K)}.
$$
\end{Proposition}

\begin{proof}
This is \cite[Theorem 2.8.]{GrKo}.
\end{proof}

To estimate certain bilinear exponential sums coming up in our method, we use the double large sieve. 

\begin{Proposition}[Double large sieve] \label{doublelargesieve} Let $\alpha_1,...,\alpha_K$ and $\beta_1,...,\beta_L$ be real numbers satisfying $|\alpha_k|\le A$ and $|\beta_l|\le B$ for $1\le k\le K$ and $1\le l\le L$. Then for any complex numbers $a_1,...,a_K$ and $b_1,...,b_L$, we have
$$
\bigg|\sum\limits_{k=1}^{K} \sum\limits_{l=1}^L a_kb_le\left(\alpha_k\beta_l\right)\bigg|\le 5(AB+1)^{1/2} \bigg(\sum\limits_{\substack{1\le k_1,k_2\le K\\ |\alpha_{k_1}-\alpha_{k_2}|<1/B}} |a_{k_1}a_{k_2}|\bigg)^{1/2} \bigg(\sum\limits_{\substack{1\le l_1,l_2\le L\\ |\beta_{l_1}-\beta_{l_2}|<1/A}} |b_{l_1}b_{l_2}|\bigg)^{1/2}.
$$
\end{Proposition}

\begin{proof} This is \cite[Theorem 7.2.]{IwKo}.\end{proof}

We also need to bound complete exponential sums over $\mathbb{Z}/r\mathbb{Z}$ with rational function entries.
We will divide them into exponential sums over $\mathbb{Z}/p^m\mathbb{Z}$, where $p$  is a prime and $p^m$ is its largest power dividing $r$. Then we consider two cases: $m=1$ and $m>1$. For the case $m=1$, we use the following bound which is a consequence of a result of Bombieri \cite{Bom}.

\begin{Proposition}[Bombieri] \label{primemoduli} Let $p$ be a prime and $f=f_1/f_2$ be a nonconstant rational function, where $f_1,f_2\in \mathbb{Z}[X]$. Let $d_p(f_1)$ and $d_p(f_2)$ be the degrees of $f_1$ and $f_2$ over $\mathbb{Z}/p\mathbb{Z}$, respectively. Let $d_p(f):=d_p(f_1)+d_p(f_2)$ be the total degree of $f$ over $\mathbb{Z}/p\mathbb{Z}$. Set 
$$
S(f,p):=\sum\limits_{\substack{n\bmod{p}\\ f_2(n)\not\equiv 0 \bmod{p}}} e_{p}\left(f(n)\right).
$$
Then 
$$
|S(f,p)|\le 2d_p(f)p^{1/2}.
$$
\end{Proposition}

\begin{proof} This follows from \cite[inequality (1.4)]{Coch}.
\end{proof}

To handle the case $m>1$, we use a result of Cochrane and Zheng below. To formulate this result, we first introduce some notations. For a polynomial $f$ with integer coefficients, we define its order $\text{ord}_p(f)$ modulo $p$ as the largest exponent $k$ such that $p^k$ divides all the coefficients of $f$. If $f=f_1/f_2\in \mathbb{Z}(X)$ with $f_1,f_2\in \mathbb{Z}[X]$ is a rational function over $\mathbb{Z}$, then we define its order modulo $p$ as $\text{ord}_p(f):=\text{ord}_p(f_1)-\text{ord}_p(f_2)$. Set $t:=\text{ord}_p(f')$. A critical point $\alpha$ of $f$ modulo $p$ is a zero of $p^{-t}f'$ over $\mathbb{Z}/p\mathbb{Z}$, i.e. $p^{-t}f'(\alpha)\equiv 0\bmod{p}$. By $\nu_p(f,\alpha)$ we denote its multiplicity. If $\alpha$ is not a critical point of $f$ modulo $p$, then we set $\nu_p(f,\alpha):=0$. 

\begin{Proposition}[Cochrane-Zheng] \label{primepowermoduli}
Let $p$ be a prime and $f=f_1/f_2$ be a nonconstant rational function, where $f_1,f_2\in \mathbb{Z}[X]$. Set $t:=\text{\rm ord}_p(f')$. Suppose that $\alpha\in \mathbb{Z}$ such that $f_2(\alpha)\not\equiv 0\bmod{p}$. Set $\nu:=\nu_p(f,\alpha)$. Let $m$ be a positive integer. Suppose that $m\ge t+2$ if $p$ is odd and $m\ge t+3$ if $p=2$. Set 
$$
S_{\alpha}(f,p^m):=\sum\limits_{\substack{n\bmod{p^m}\\ n\equiv \alpha\bmod{p}}} e_{p^m}\left(f(n)\right).
$$  
Then $S_\alpha(f,p^m)=0$ if $\nu=0$ and 
$$
\left|S_\alpha(f,p^m)\right|\le \nu p^{t/(\nu+1)}p^{m(1-1/(\nu+1))} 
$$
if $\nu\ge 1$. 
\end{Proposition}

\begin{proof} This is found in \cite[Theorem 3.1]{Coch}.
\end{proof}

At some point, we will have to evaluate quadratic Gauss sums. We recall the following result.

\begin{Proposition}[Quadratic Gauss sums]\label{Gausssums}
For $q\in\mathbb{N}$ and $a,b\in \mathbb{Z}$, define
$$
G(q;a,b):=\sum\limits_{n=1}^q e_q\left(an^2+bn\right).
$$
Suppose that $q$ is odd.
Then
$$
G(q;a,b)=0 \quad \mbox{ if } (a,q)\nmid b
$$
and 
$$
G(q;a,b)=\epsilon_q \cdot e_q\left(-\overline{4a}b^2\right)\cdot \left(\frac{a}{q}\right)\cdot \sqrt{q} \quad \mbox{ if } (a,q)=1, 
$$
where $\left(\frac{\cdot}{q}\right)$ is the Jacobi symbol and
$$
\epsilon_q:=\begin{cases} 1 & \mbox{ if } q\equiv 1\bmod{4},\\ i & \mbox{ if } q\equiv 3\bmod{4}.\end{cases}
$$
\end{Proposition}

\begin{proof} This follows from the results in \cite[subsection 7.4]{GrKo} and \cite[Theorem 1.1.5 and Lemma 1.2.1]{BEW}.
\end{proof}

Finally, we will need the following standard estimate for sums of greatest common divisors.

\begin{Lemma} \label{gcdsums}
Let $0<\sigma\le 1$, $H\ge 1$ and $r\in \mathbb{N}$. Then
$$
\sum\limits_{1\le h\le H} (h,r)^{\sigma}\ll_{\varepsilon} Hr^{\varepsilon}.
$$
\end{Lemma}

\begin{proof} We divide the sum over $h$ according to the value of $(r,h)$, getting
$$
\sum\limits_{1\le h\le H} (h,r)^{\sigma}=\sum\limits_{d|r} d^{\sigma}\sum\limits_{\substack{1\le h\le H\\ (h,r)=d}} 1\le \sum\limits_{d|r} d^{\sigma}\sum\limits_{\substack{1\le h\le H\\ d|h}} 1\le H\sum\limits_{d|r} d^{\sigma-1}\le H\tau(r)\ll_{\varepsilon} Hr^{\varepsilon}.
$$
\end{proof}

\section{Wolke's approach} \label{WolApp}
In this section, we describe Wolke's approach to the case of prime moduli and its extension to square moduli. 

\subsection{General approach}
It is easy to prove (see \cite[Lemma 1]{BaiZhao1}) that for any finite sequence $\alpha_1,...,\alpha_K$ of real numbers, one has a bound of the form
\begin{equation} \label{lsreduce}
\sum\limits_{k=1}^K \left|\sum\limits_{M<n\le M+N} a_ne(n\alpha_k)\right|^2 \ll 
\max\limits_{x\in \mathbb{R}} \sharp\left\{k\in \{1,...,K\} : ||\alpha_k-x||\le \frac{1}{N}\right\}\cdot NZ,
\end{equation}
where for any real number $z$, $||z||$ denotes the distance of $z$ to the nearest integer. Throughout the sequel,
we will write
\begin{equation} \label{Deltadef}
\Delta:=\frac{1}{N}.
\end{equation}
As a consequence, for any subset $\mathcal{S}$ of the positive integers,  \eqref{lssparse} holds for
$$
\Delta_{\mathcal{S}}(Q,N)=N\cdot \max\limits_{x\in [0,1]} \sharp \left\{\frac{a}{q} : q\in \mathcal{S},\ q\le Q,\ 1\le a\le q,\ (a,q)=1,\ \left|\left|\frac{a}{q}-x\right|\right|\le \Delta\right\}.
$$ 
This links large sieve bounds with bounds for the number of Farey fractions in short intervals. 

\subsection{Wolke's treatment of prime moduli} 
For the case when $\mathcal{S}=\mathbb{P}$ is the set of primes, one therefore needs to estimate the quantity
$$
M:= \max\limits_{x\in [0,1]} \sharp \left\{\frac{a}{p} : p\in \mathbb{P},\ p\le Q,\ 1\le a\le p-1,\ \left|\left|\frac{a}{p}-x\right|\right|\le \Delta\right\}.
$$ 
In his treatment of prime moduli, Wolke \cite{Wol} used the Dirichlet approximation theorem to approximate $x\in [0,1]$ above by a fraction $b/r$ so that 
\begin{equation} \label{appro}
1\le r\le \tau, \quad (b,r)=1, \quad x=\frac{b}{r}+z \mbox{ with } |z|\le \frac{1}{r\tau}.
\end{equation}
Here and in the sequel, we set
$$
\tau=\left[\sqrt{N}\right]=\left[\frac{1}{\sqrt{\Delta}}\right].
$$
(With this choice of $\tau$, we have $1/(r\tau)\ge \Delta$ if $1\le r\le \tau$, which is essential in Wolke's method.) 
Now it turns out that counting Farey fractions $a/p$ in small intervals of the form 
$$
[x-\Delta,x+\Delta]=\left[\frac{b}{r}+z-\Delta,\frac{b}{r}+z+\Delta\right]
$$ 
relates to counting primes $p$ in small segments of residue classes modulo $r$. This can be done using the powerful Brun-Titchmarsh theorem. Along these lines, Wolke obtained a bound of the form
$$
\sum\limits_{\substack{p\in \mathbb{P}\\ p\le Q}} \sum\limits_{a=1}^{p-1} \left|\sum\limits_{M<n\le M+N} a_ne\left(\frac{na}{p}\right)\right|^2\ll \frac{Q^2\log\log Q}{\log(Q/\sqrt{N})}\cdot Z
$$ 
if $\sqrt{N}<Q\le N$, as mentioned in subsection \ref{knownresults}. 

\subsection{Application of Wolke's approach to square moduli} It turns out that Wolke's approach also gives fairly good large sieve bounds for the case of square moduli. In this case, one needs an upper bound for the quantity
$$
\tilde{M}:=\max\limits_{x\in [0,1]} P(x),
$$
where 
\begin{equation} \label{Pxdef}
P(x):= \sharp \left\{\frac{a}{q^2} : Q<q\le 2Q,\ 1\le a\le q,\ (a,q)=1,\ \left|\left|\frac{a}{q^2}-x\right|\right|\le \Delta\right\}.
\end{equation}
Note that for convenience, above we have restricted $q$ to a dyadic interval. 
Using \eqref{lsreduce} and a division of the $q$-range into such dyadic intervals, we have the following. If a bound of the form
$$
\tilde{M}\ll F(Q,N)
$$
with a function $F(Q,N)$ that is non-decreasing in $Q$ holds, then \eqref{Zhaobound} with 
\begin{equation} \label{tildeDeltaconcrete}
\tilde{\Delta}(Q,N)=NF(Q,N)\log 2Q
\end{equation}
follows. 
To bound $P(x)$, we again approximate $x$ in the form in \eqref{appro}. It is easy to see that one may assume $|z|\ge \Delta$ without loss of generality. Moreover, the case of negative $z$ works in a similar way as that of positive $z$. So we assume that 
\begin{equation} \label{zrange}
\Delta\le z\le \frac{1}{r\tau}=\frac{1}{r\left[\sqrt{N}\right]}=\frac{1}{r\left[1/\sqrt{\Delta}\right]}
\end{equation}
throughout the sequel. In \cite{Bai1} and \cite{BaiZhao1} we started by bounding $P(x)$ in a convenient way using smooth weight functions. This can be formulated in the following form (see \cite[section 5]{BaiZhao1}, in particular, with some modifications).

\begin{Proposition} \label{iniPx}
Assume that the condition \eqref{QNcond} is satisfied. 
Assume also that $x\in [0,1]$ satisfies an approximation of the form in \eqref{appro} and $z$ lies in the range in \eqref{zrange}. Then
\begin{equation} \label{Pxintbound}
P(x)\ll 1+\delta^{-1}\int\limits_{\mathbb{R}} \Psi\left(\frac{y}{Q^2}\right)\Omega(\delta,y){\rm d}y.
\end{equation}
Here $\delta$ is any parameter satisfying
\begin{equation} \label{deltarange}
\frac{Q^2\Delta}{z}\le \delta\le Q^2,
\end{equation}
$\Psi$ with $0\le \Psi\le 1$ is a Schwartz class function with compact support $\mathcal{C}$ in the range
$$
\left[\frac{3}{4},\frac{9}{4}\right]\subseteq \mathcal{C}\subseteq \left[\frac{1}{2},\frac{5}{2}\right],
$$
and
\begin{equation} \label{Omegadef}
\Omega(\delta,y):=\sum\limits_{q\in \mathbb{Z}} \Phi_1\left(\frac{q-\sqrt{y}}{\delta/Q}\right)\sum\limits_{\substack{m\in \mathbb{Z}\\ m\equiv -bq^2\bmod{r}}} \Phi_2\left(\frac{m-yrz}{\delta rz}\right),
\end{equation}
$\Phi_1$ and $\Phi_2$ with $0\le \Phi_1,\Phi_2\le 1$ being suitable Schwartz class functions with compact support in 
$\mathbb{R}_{>0}$.  
\end{Proposition}  

Essentially, this means is that $q$ runs over a short interval of length $\delta/Q$ around $\sqrt{y}$, $m$ runs over a short interval of length $\delta rz$ around $yrz$, these two variables are linked by the congruence $m\equiv -bq^2\bmod{r}$, and the count of pairs $(m,q)$ in the resulting small box is averaged over the variable $y$ of size $Q^2$. This average is scaled by a factor of $\delta^{-1}$. Note that with increasing $\delta$, the area of the box increases by a factor of $\delta^2$. So it appears that taking $\delta$ as small as possible in \eqref{deltarange} is optimal. However, after Fourier analytic transformations, it may be better to choose $\delta$ differently. So it helps to have flexibility in the choice of $\delta$. Writing the congruence in the form $q^2\equiv -\overline{b}m \bmod{r}$, the task becomes to count squares in small segments of residue classes modulo $r$, which is a situation similar to that considered by Wolke for primes.     

\section{Previous ways to estimate $P(x)$}
In this section, we review how the quantity $P(x)$ was estimated in \cite{Bai1} and \cite{BaiZhao1}, indicate advantages and disadvantages of these methods and describe how we plan to overcome the disadvantages in this article. 
\subsection{Direct count} A simple way to estimate $P(x)$ is to directly count pairs $(m,q)$ satisfying the congruence
$$
q^2\equiv -\overline{b}m\bmod r
$$
in the boxes described in the previous subsection, followed by an averaging over $y$ in a range of size $\asymp Q^2$. It is easy to work out that for a given $m$, we get
$$
O\left(\left(1+\frac{\delta/Q}{r}\right)r^{\varepsilon}\right)
$$
possible $q$'s in the said box. The number of possible $m$'s is clearly bounded by 
$$
O\left(1+\delta rz\right). 
$$
This gives 
\begin{equation} \label{firstbound}
P(x)\ll 1+\delta^{-1}Q^2\left(1+\delta rz\right)\left(1+\frac{\delta/Q}{r}\right)r^{\varepsilon}.
\end{equation}
Following the method in \cite{Bai1}, it turns out that the averaging over $y$ allows one to remove the term "$1+$" in "$1+\delta rz$" above. Consequently, we  get
$$
P(x)\ll 1+\left(Q^2rz+\delta Qz\right)r^{\varepsilon}.
$$
Choosing $\delta$ as small as possible in \eqref{deltarange}, i.e. $\delta=Q^2\Delta/z$, it follows that
\begin{equation} \label{fromtrivial0}
P(x)\ll 1+\left(Q^2rz+Q^3\Delta\right)r^{\varepsilon}
\end{equation}
and hence 
\begin{equation} \label{fromtrivial}
P(x)\ll \left(1+Q^2\sqrt{\Delta}+Q^3\Delta\right)N^{\varepsilon},
\end{equation}
which yields the large sieve bound \eqref{Zhaobound} with 
\begin{equation} \label{lstrivial}
\tilde{\Delta}(Q,N)=\left(N+Q^2\sqrt{N}+Q^3\right)N^{\varepsilon}
\end{equation}
achieved in \cite{Bai1}. We note that using $\Delta=1/N$ and our condition \eqref{QNcond}, the bound \eqref{fromtrivial} simplifies into
\begin{equation} \label{Pxsimpler}
P(x)\ll Q^2\Delta^{1/2}N^{\varepsilon}.
\end{equation}
This is the estimate which we seek to beat in the sequel. In what follows, we will assume that 
\begin{equation} \label{fundcond}
Q^2rz\ge N^{\varepsilon}.
\end{equation}
If this is not satisfied, then \eqref{fromtrivial0} yields the estimate
\begin{equation} \label{bestposs}
P(x)\ll \left(1+Q^3\Delta\right)N^{2\varepsilon},
\end{equation}
which is the best possible we can hope for and consistent with Zhao's conjecture \eqref{conj}. Let us record the established bound \eqref{fromtrivial0} in a slightly modified form as an extra lemma below.

\begin{Lemma} \label{previousPxbound}
Under the condition \eqref{QNcond}, we have
$$
P(x)\ll (1+Q^2rz+Q^3\Delta)N^{\varepsilon}.
$$
\end{Lemma}

\begin{proof}
This can be found in \cite[Lemma 6]{Bai1}. 
\end{proof}

\subsection{Poisson summation in $m$ and $q$} One may make progress on the problem of bounding $P(x)$ by using Fourier analytic tools. In \cite{Bai1}, the double sum over $q$ and $m$ in \eqref{Omegadef} was transformed by applying the Poisson summation formula in both variables. This led to an expression involving quadratic Gauss sums and exponential integrals. Bounding them from above gave the estimate (see \cite[inequality (49)]{Bai1})
$$
P(x)\ll \left(1+\delta Qz+\frac{\delta}{Qr^{1/2}}+\frac{Qr^{1/2}}{\delta^{1/2}}\right)N^{\varepsilon},
$$
which yielded an improvement of the large sieve bound in \eqref{lstrivial} if $Q\ge N^{5/14+\varepsilon}$ and established Zhao's conjecture \eqref{conj} if $Q\ge N^{5/12+\varepsilon}$. In \cite{BaiZhao1}, this approach was developed further. The authors evaluated the said quadratic Gauss sums and exponential integrals explicitly, which resulted in a dual expression. Next, they applied Weyl differencing, which led to a new counting problem. This gave an improvement if $r\le Q$. For $r>Q$, they applied a more direct Fourier analytic approach following \cite{Zhao1}, which also involved Poisson summation and Weyl differencing. Altogether, they obtained a bound of 
$$
P(x)\ll \left(Q^3\Delta+Q^{1/2}\right)(QN)^{\varepsilon}
$$  
if $Q\ge N^{1/3}$, giving \eqref{Zhaobound} with 
\begin{equation} \label{lsimp}
\tilde{\Delta}(Q,N)=\left(Q^3+Q^{1/2}N\right)(QN)^{\varepsilon}.
\end{equation}
Combining this with \eqref{lstrivial}, which is stronger if $Q\le N^{1/3}$,  implies \eqref{Zhaobound} with $\tilde{\Delta}(Q,N)$ as in \eqref{tildeDelta'}.

\subsection{Advantages and disadvantages} A direct count is simple but unable to avoid the contribution "$1+$" in the term 
$$
1+\frac{\delta/Q}{r}
$$
appearing in  \eqref{firstbound}, which leads to the extra term $Q^2N^{1/2}$ on the right-hand side of \eqref{lstrivial}, as compared to Zhao's conjecture \eqref{conj}. The situation may improve if we apply Fourier analytic tools. However, a double application of Poisson summation in both variables $q$ and $m$ results in complete exponential sums modulo $r$ (in this case, quadratic Gauss sums). A bound of $\ll r^{1/2}$ is the best we can hope for when estimating a complete exponential sum modulo $r$ (except this exponential sum vanishes). This creates a barrier which can be suboptimal if at least one of the initial sums over $q$ and $m$ is short, as compared to the modulus $r$. In particular, in the case when $N=Q^3$ and $r$ and $z$ are as large as possible (i.e., $r=Q^{3/2}$ and $z=1/Q^{3}$), it turns out that $m\asymp r^{1/3}$ and $q\asymp r^{2/3}$. Since $r^{1/3}$ is much smaller and $r^{2/3}$ is much larger than the square root of the modulus $r$, Poisson summation will be effective in the variable $q$ but wasteful in the variable $m$. Consequently, in this article, we overcome both of the issues raised above by only applying Poisson summation in $q$ and then arguing differently than in \cite{BaiZhao1}. Here, the distribution of square roots modulo $r$ will come into the picture.      

\section{Poisson summation only in $q$}
\subsection{Modification of the basic estimate}
In the following, we modify our bound for $P(x)$ in Proposition \ref{iniPx} in such a way that we can conveniently apply Poisson summation in the variable $q$ and calculate Fourier integrals arising from the integration in \ref{iniPx}. First, we replace the integration variable $y$ on the right-hand side of \eqref{Pxintbound} by $w^2$, thus obtaining 
\begin{equation} \label{mo}
P(x)\ll 1+\delta^{-1}\int\limits_{\mathbb{R}} 2w\Psi\left(\left(\frac{w}{Q}\right)^2\right)\Omega(\delta,w^2){\rm d}w\ll  1+\delta^{-1}Q\int\limits_{\mathbb{R}} \tilde{\Psi}\left(\frac{w}{Q}\right)\Omega(\delta,w^2){\rm d}w,
\end{equation}
where  $\tilde{\Psi}(u):=\Psi(u^2)$ for all $u\in \mathbb{R}$. So $w\asymp Q$ above. By the definition of $\Omega(\delta,y)$ in \eqref{Omegadef}, we have  
\begin{equation*}
\Omega(\delta,w^2):=\sum\limits_{q\in \mathbb{Z}} \Phi_1\left(\frac{q-w}{\delta/Q}\right)\sum\limits_{\substack{m\in \mathbb{Z}\\ m\equiv -bq^2\bmod{r}}} \Phi_2\left(\frac{m-w^2rz}{\delta rz}\right).
\end{equation*}
Recalling \eqref{deltarange}, we also see that $m\asymp w^2rz \asymp Q^2rz$ above.  
Moreover, we observe that  
$$
\frac{m-w^2rz}{\delta rz}=\frac{\sqrt{m}-w\sqrt{rz}}{(\delta rz)/(\sqrt{m}+w\sqrt{rz})} \asymp  \frac{\sqrt{m}-w\sqrt{rz}}{(\delta rz)/(w\sqrt{rz})} \asymp  \frac{\sqrt{m}-w\sqrt{rz}}{\delta\sqrt{rz}/Q}.
$$
Considering this and dropping the term $ \tilde{\Psi}\left(w/Q\right)$ in \eqref{mo}, we may replace the bound \eqref{Pxintbound} by 
\begin{equation} \label{Pxmodified}
P(x)\ll 1+\delta^{-1}Q\int\limits_{\mathbb{R}} \tilde{\Omega}(\delta,w){\rm d}w, 
\end{equation}
where
\begin{equation} \label{Omegamodified}
\tilde{\Omega}(\delta,w):=\sum\limits_{q\in \mathbb{Z}} \Phi_3\left(\frac{q-w}{\delta/Q}\right)\sum\limits_{\substack{m\in \mathbb{Z}\\ m\equiv -bq^2\bmod{r}}} \Phi_4\left(\frac{\sqrt{m}-w\sqrt{rz}}{\delta \sqrt{rz}/Q}\right) \Phi_5\left(\frac{m}{Q^2rz}\right)
\end{equation}
for suitable Schwartz class functions $\Phi_3$, $\Phi_4$ and $\Phi_5$ satisfying $0\le \Phi_i\le 1$ and having compact support in $\mathbb{R}_{>0}$. 
Furthermore, to apply Poisson summation in the variable $q$, we rewrite the congruence $m\equiv -bq^2\bmod{r}$ above  in the form
$$
q\equiv \sqrt{jm}\bmod{r} \quad \mbox{with } j:=-\overline{b},
$$
provided a square root $\sqrt{jm}$ of $jm$ modulo $r$ exists. (We recall that we really mean the collection of {\it all} modular square roots of $jm$ modulo $r$ rather than a single one.) Now re-arranging the summations on the right-hand side of \eqref{Omegamodified}, we get 
\begin{equation} \label{tildeOmega}
\tilde{\Omega}(\delta,w)=\sum\limits_{m\in \mathbb{Z}}\Phi_5\left(\frac{m}{Q^2rz}\right) \Phi_4\left(\frac{\sqrt{m}-w\sqrt{rz}}{\delta \sqrt{rz}/Q}\right) \sum\limits_{\substack{q\in \mathbb{Z}\\ q\equiv \sqrt{jm}\bmod{r}}}  \Phi_3\left(\frac{q-w}{\delta/Q}\right).
\end{equation}
We point out that the square root of $m$ appearing in the argument of $\Phi_4$ above is the usual analytic square root of $m$, whereas the square root of $jm$ appearing in the summation condition for the variable $q$ is a modular square root.  

\subsection{Poisson summation and evaluation of Fourier integrals}
Applying Proposition \ref{Poisum} to the inner-most sum over $q$ in \eqref{tildeOmega} gives
$$
 \sum\limits_{\substack{q\in \mathbb{Z}\\ q\equiv \sqrt{jm}\bmod{r}}}  \Phi_3\left(\frac{q-w}{\delta/Q}\right)=\frac{\delta}{Qr}\cdot \sum\limits_{l\in \mathbb{Z}} \hat{\Phi}_3\left(\frac{l\delta}{Qr}\right)e_r\left(l(\sqrt{jm}-w)\right),
$$
where again the square root in the exponential denotes a modular square root modulo $r$. 
Plugging this into \eqref{tildeOmega}, re-arranging the $m$- and $l$-summations, using \eqref{Pxmodified} and pulling in the integral, we have
$$
P(x)\ll 1+\frac{1}{r} \sum\limits_{l\in \mathbb{Z}} \hat{\Phi}_3\left(\frac{l\delta}{Qr}\right)\sum\limits_{m\in \mathbb{Z}}\Phi_5\left(\frac{m}{Q^2rz}\right)  e_r\left(l\sqrt{jm}\right) \int\limits_{\mathbb{R}}\Phi_4\left(\frac{\sqrt{m}-w\sqrt{rz}}{\delta \sqrt{rz}/Q}\right) e\left(-\frac{l}{r}\cdot w\right){\rm d}w. 
$$ 
The Fourier integral above equals
$$
\int\limits_{\mathbb{R}}\Phi_4\left(\frac{\sqrt{m}-w\sqrt{rz}}{\delta \sqrt{rz}/Q}\right) e\left(-\frac{l}{r}\cdot w\right){\rm d}w=-\frac{\delta}{Q} \cdot \hat{\Phi}_4\left(-\frac{l\delta}{Qr}\right)e\left(-\frac{l\sqrt{m}}{r^{3/2}z^{1/2}}\right). 
$$
We conclude that
\begin{equation} \label{afterPoisson}
P(x)\ll 1+\frac{\delta}{Qr} \left|\sum\limits_{l\in \mathbb{Z}} W\left(\frac{l\delta}{Qr}\right)\sum\limits_{m\in \mathbb{Z}}\Phi_5\left(\frac{m}{Q^2rz}\right) e\left(-\frac{l\sqrt{m}}{r^{3/2}z^{1/2}}\right) e_r\left(l\sqrt{jm}\right)\right|, 
\end{equation}
where $W(u):=\hat{\Phi}_3(u)\hat{\Phi}_4(-u)$ for all $u\in \mathbb{R}$.

\subsection{Discussion of the new bound for $P(x)$}
Bounding the sum on the right-hand side of \eqref{afterPoisson} trivially gives
$$
P(x)\ll 1+\frac{\delta}{Qr}\cdot \left(1+\frac{Qr}{\delta}\right)\left(1+Q^2rz\right)=1+\frac{\delta}{Qr}+Q\delta z+Q^2rz.
$$
Taking $\delta$ as small as possible, i.e. $\delta:=Q^2\Delta/z$, it follows that
$$
P(x)\ll 1+\frac{Q\Delta}{rz}+Q^3\Delta+Q^2rz\ll 1+\frac{Q\Delta}{rz}+Q^3\Delta+Q^2\sqrt{\Delta}.
$$ 
If $\sqrt{\Delta}/\sqrt{Q}\le rz\le \sqrt{\Delta}$, then the above bound for $P(x)$ is by a factor of $N^{\varepsilon}$ stronger than \eqref{fromtrivial}. If $rz<\sqrt{\Delta}/\sqrt{Q}$, then Lemma \ref{previousPxbound} implies 
$$
P(x)\ll \left(1+Q^{3/2}\Delta^{1/2}+Q^3\Delta\right)N^{\varepsilon}\ll  \left(1+Q^3\Delta\right)N^{\varepsilon},
$$
which is the expected bound for $P(x)$, consistent with Zhao's conjecture \eqref{conj}. We therefore aim to exploit cancellations in the sum on the right-hand side of \eqref{afterPoisson} to obtain a non-trivial bound for this sum. Since \eqref{fromtrivial} (respectively, \eqref{Pxsimpler}) is the best we have so far in the range $N^{1/4}\le Q\le N^{1/3}$, any improvement of this bound by a factor of $Q^{-\eta}$ for some $\eta>0$ would be a substantial progress on the problem.  

\subsection{Focus on a particular situation} As remarked in subsection \ref{known}, there has been no progress at the point $Q=N^{1/3}$ since Zhao's initial result in \cite{Zhao1} on the large sieve for square moduli. It is therefore of importance to pay special attention to this particular point. Moreover, the worst case in the estimate in Lemma \ref{previousPxbound} occurs when $rz$ is as large as possible, i.e. $rz=\Delta^{1/2}$. Note that if $Q=N^{1/3}$, then $\Delta^{1/2}=N^{-1/2}=Q^{-3/2}$. For this reason, we should keep track of what happens if 
\begin{equation} \label{specialsituation}
Q=N^{1/3} \quad \mbox{and} \quad rz=Q^{-3/2}.
\end{equation}
In this case, our bound for $P(x)$ in \eqref{afterPoisson} turns into
\begin{equation} \label{specialPx}
P(x)\ll 1+\frac{\delta}{Qr} \left|\sum\limits_{l\in \mathbb{Z}} W\left(\frac{l\delta}{Qr}\right)\sum\limits_{m\in \mathbb{Z}}\Phi_5\left(\frac{m}{Q^{1/2}}\right) e\left(-\frac{l\sqrt{m}}{r}\cdot Q^{3/4}\right) e_r\left(l\sqrt{jm}\right)\right|. 
\end{equation}
The following sections are organized as follows. We successively treat different ranges of $r$ by different means. For each range, we first investigate the special situation in which the equations in \eqref{specialsituation} hold and then handle the general case. Our investigation of the special situation will be sketchy, omitting some technical details, and the treatment of the general case will be precise. This procedure will help us to gain an intuition of what needs to be done before carrying out the calculations in full detail.

\section{Handling large $r$'s}
\subsection{Sketch for our particular situation} \label{sketchlarge} Let us look first at the situation when the equations in \eqref{specialsituation} are satisfied and $r$ is as large as possible, i.e. $r=N^{1/2}=Q^{3/2}$. In this case, taking \eqref{zrange} and \eqref{deltarange} into account, we necessarily have $z=\Delta$ and $\delta=Q^2$,
and \eqref{specialPx} turns into
\begin{equation} \label{specialPxlarger}
P(x)\ll 1+\frac{1}{Q^{1/2}} \left|\sum\limits_{l\in \mathbb{Z}} W\left(\frac{l}{Q^{1/2}}\right)\sum\limits_{m\in \mathbb{Z}}\Phi_5\left(\frac{m}{Q^{1/2}}\right) e\left(-\frac{l\sqrt{m}}{Q^{3/4}}\right) e_r\left(l\sqrt{jm}\right)\right|. 
\end{equation}
Bounding the right-hand side trivially gives
$$
P(x)\ll Q^{1/2},
$$
which is, up to a factor of $N^{\varepsilon}$, the same as \eqref{Pxsimpler}.
Our goal is to beat this bound, i.e. to show that $P(x)\ll Q^{1/2-\eta}$ for some $\eta>0$. Since $W$ has rapid decay and $\Phi_5$ has compact support in $\mathbb{R}_{>0}$, the right-hand side of \eqref{specialPxlarger} behaves essentially like 
\begin{equation} \label{essential1}
1+\frac{1}{Q^{1/2}} \left| \sum\limits_{|l|\ll Q^{1/2}}   \sum\limits_{m\asymp Q^{1/2}} e\left(-\frac{l\sqrt{m}}{Q^{3/4}}\right) e_r\left(l\sqrt{jm}\right)\right|.
\end{equation}
In the following, we call the exponential
\begin{equation} \label{analyticterm}
e\left(-\frac{l\sqrt{m}}{Q^{3/4}}\right)
\end{equation}
above the "analytic term" and the exponential 
\begin{equation} \label{arithmeticterm}
e_r\left(l\sqrt{jm}\right)
\end{equation}
the "arithmetic term". We observe that under the summation conditions on $l$ and $m$ in \eqref{essential1}, we have 
$$
\frac{l\sqrt{m}}{Q^{3/4}}\ll 1,
$$
and hence the analytic term does not oscillate. So we can remove it by partial summation without costs. It remains to bound the sum
$$
\Sigma=\sum\limits_{|l|\ll L}   \sum\limits_{m\asymp Q^{1/2}} e_r\left(l\sqrt{jm}\right)
$$ 
with $L\le Q^{1/2}$ nontrivially by $O(Q^{1-\eta})$ for some $\eta>0$. The above is an incomplete exponential sum in two variables modulo $r$. Completion techniques (such as Poissson summation) will not be efficient if the ranges of the variables are much shorter than $r^{1/2}$. In our case, the summation ranges are of size about $r^{1/3}$, which is very short. In this case, a reasonable idea is to create new points using the Cauchy-Schwarz inequality. This gives
\begin{equation} \label{postCS}
|\Sigma|^2\ll Q^{1/2} \sum\limits_{|l|\ll Q^{1/2}}  \left| \sum\limits_{m\asymp Q^{1/2}} e_r\left(l\sqrt{jm}\right)\right|^2\ll Q^{1/2}  \sum\limits_{l\in \mathbb{Z}}  V\left(\frac{l}{Q^{1/2}}\right) \left| \sum\limits_{m\asymp Q^{1/2}} e_r\left(l\sqrt{jm}\right)\right|^2,
\end{equation}
where $V$ is a suitable Schwartz class function satisfying $0\le V\le 1$. Expanding the modulus square and pulling in the sum over $l$, we deduce that
$$
|\Sigma|^2\ll Q^{1/2} \sum\limits_{m_1,m_2\asymp Q^{1/2}} \sum\limits_{l\in \mathbb{Z}} V\left(\frac{l}{Q^{1/2}}\right) e_r\left(l(\sqrt{jm_1}-\sqrt{jm_2})\right).
$$
Applying the Poisson summation formula, Proposition \ref{Poisum} with $L=r/Q^{1/2}=Q$ and $M=0$, it follows that
$$
|\Sigma|^2\ll Q\sum\limits_{m_1,m_2\asymp Q^{1/2}} \sum\limits_{\substack{n\in \mathbb{Z}\\ n\equiv\sqrt{jm_1}-\sqrt{jm_2} \bmod{r} }} \hat{V}\left(\frac{Q^{1/2}n}{r}\right).
$$
Using the rapid decay of $\hat{V}$, essentially only $n$'s with $n\ll r/Q^{1/2}$ contribute significantly, and hence we essentially get a bound of
\begin{equation} \label{Sigmabound}
|\Sigma|^2\ll Q \sum\limits_{\substack{m_1,m_2\asymp Q^{1/2}\\ \left|\left|(\sqrt{jm_1}-\sqrt{jm_2})/r\right|\right|\ll Q^{-1/2}}} 1.
\end{equation}
The same can be achieved via an application of the double large sieve, Proposition \ref{doublelargesieve}, with the convention that the modular square root $\sqrt{jm}$ lies in the range $0\le \sqrt{jm}<r$ and so $0\le \sqrt{jm}/r<1$. 
Discarding the summation condition $\left|\left|(\sqrt{jm_1}-\sqrt{jm_2})/r\right|\right|\ll Q^{-1/2}$, we get the trivial bound $|\Sigma|^2\ll Q^2$. Hence, we need to make use of this condition to beat the trivial bound. Heuristically, we expect that the "event" $\left|\left|(\sqrt{jm_1}-\sqrt{jm_2})/r\right|\right|\ll Q^{-1/2}$ occurs with a probability of $Q^{-1/2}$ when picking $m_1$ and $m_2$ randomly. Trivially, this event occurs when $m_1=m_2$. So it is reasonable to conjecture that 
$$
|\Sigma|^2\ll Q^{3/2}.
$$
However, it is difficult to detect the said event. A natural idea is to write the sum on the right-hand side of \eqref{Sigmabound} as 
$$
\Sigma'=\sum\limits_{\substack{m_1,m_2\asymp Q^{1/2}\\ \left|\left|(\sqrt{jm_1}-\sqrt{jm_2})/r\right|\right|\ll Q^{-1/2}}} 1
= \sum\limits_{\substack{\lambda\in \mathbb{Z}\\ |\lambda|\ll Q^{-1/2}r}} \sum\limits_{\substack{m_1,m_2\asymp Q^{1/2}\\ \sqrt{jm_1}-\sqrt{jm_2}\equiv \lambda \bmod{r}}} 1
$$
and apply Cauchy-Schwarz again to get
\begin{equation*}
\begin{split}
|\Sigma'|^2\ll & \frac{r}{Q^{1/2}}  \sum\limits_{\substack{\lambda\in \mathbb{Z}\\ |\lambda|\ll Q^{-1/2}r}} \bigg| \sum\limits_{\substack{m_1,m_2\asymp Q^{1/2}\\ \sqrt{jm_1}-\sqrt{jm_2}\equiv \lambda \bmod{r}}} 1\bigg|^2\\
\le & Q \sum\limits_{\lambda\bmod{r}} \bigg| \sum\limits_{\substack{m_1,m_2\asymp Q^{1/2}\\ \sqrt{jm_1}-\sqrt{jm_2}\equiv \lambda \bmod{r}}} 1\bigg|^2\\
= & Q  \sum\limits_{\substack{m_1,m_2,m_3,m_4\asymp Q^{1/2}\\ \sqrt{jm_1}-\sqrt{jm_2}\equiv \sqrt{jm_3}-\sqrt{jm_4} \bmod{r}}} 1\\
\ll & Q  E_2(cQ^{1/2};j,r),
\end{split}
\end{equation*}
and hence
$$
|\Sigma|^4\ll Q^3E_2(cQ^{1/2};j,r),
$$
where $c>0$ is a suitable constant and $E_2\left(cQ^{1/2};j,r\right)$ is the additive energy defined in \eqref{AE}. Of course, the above extension of the sum over $\lambda$ to all residue classes modulo $r$ comes with a cost: Now applying the conjectural bound in \eqref{2expect} only yields the trivial bound
$$
|\Sigma|^4\ll Q^{4+\varepsilon}. 
$$
So we need to work harder to obtain a saving. We may try to relate our original sum $\Sigma$ to a higher additive energy. This can be achieved by an initial application of  H\"older's inequality with the exponents $4/3$ and $4$ (or twice Cauchy-Schwarz), and then using Poisson summation and Cauchy-Schwarz again, as above. H\"older's inequality (or Cauchy-Schwarz applied to \eqref{postCS}) gives 
$$
|\Sigma|^4\ll Q^{3/2} \sum\limits_{|l|\ll Q^{1/2}}  \left| \sum\limits_{m\asymp Q^{1/2}} e_r\left(l\sqrt{jm}\right)\right|^4\ll Q^{3/2}  \sum\limits_{l\in \mathbb{Z}}  V\left(\frac{l}{Q^{1/2}}\right) \left| \sum\limits_{m\asymp Q^{1/2}} e_r\left(l\sqrt{jm}\right)\right|^4.
$$ 
Expanding the fourth power above and pulling in the sum over $l$, we deduce that
$$
|\Sigma|^4\ll Q^{3/2} \sum\limits_{m_1,m_2,m_3,m_4\asymp Q^{1/2}} \sum\limits_{l\in \mathbb{Z}} V\left(\frac{l}{Q^{1/2}}\right) e_r\left(l(\sqrt{jm_1}+\sqrt{jm_2}-\sqrt{jm_3}-\sqrt{jm_4})\right).
$$
Applying the Poisson summation formula, Proposition \ref{Poisum} with $L=r/Q^{1/2}=Q$ and $M=0$, as above, it follows that
$$
|\Sigma|^4\ll Q^2\sum\limits_{m_1,m_2,m_3,m_4\asymp Q^{1/2}} \sum\limits_{\substack{n\in \mathbb{Z}\\ n\equiv\sqrt{jm_1}+\sqrt{jm_2}-\sqrt{jm_3}-\sqrt{jm_4} \bmod{r} }} \hat{V}\left(\frac{Q^{1/2}n}{r}\right).
$$
Using the rapid decay of $\hat{V}$, we now get essentially  
\begin{equation*}
|\Sigma|^4\ll Q^2\sum\limits_{\substack{\lambda\in \mathbb{Z}\\ |\lambda|\ll Q^{-1/2}r}} \sum\limits_{\substack{m_1,m_2,m_3,m_4\asymp Q^{1/2}\\ \sqrt{jm_1}+\sqrt{jm_2}-\sqrt{jm_3}-\sqrt{jm_4}\equiv \lambda \bmod{r}}} 1.
\end{equation*}
Applying Cauchy-Schwarz again, in a similar fashion as above, we deduce that
\begin{equation*}
\begin{split}
|\Sigma|^8\ll & Q^5  \sum\limits_{\lambda\bmod{r}} \bigg| \sum\limits_{\substack{m_1,m_2,m_3,m_4\asymp Q^{1/2}\\ \sqrt{jm_1}+\sqrt{jm_2}-\sqrt{jm_3}-\sqrt{jm_4}\equiv \lambda \bmod{r}}} 1\bigg|^2\\
= & Q^5  \sum\limits_{\substack{m_1,m_2,m_3,m_4,m_5,m_6,m_7,m_8\asymp Q^{1/2}\\ \sqrt{jm_1}+\sqrt{jm_2}-\sqrt{jm_3}-\sqrt{jm_4}\equiv \sqrt{jm_5}+\sqrt{jm_6}-\sqrt{jm_7}-\sqrt{jm_8} \bmod{r}}} 1\\
\ll & Q^5  E_4(cQ^{1/2};j,r),
\end{split}
\end{equation*}
where $c>0$ is a suitable constant and $E_4\left(cQ^{1/2};j,r\right)$ is the higher additive energy defined in \eqref{AE4}. Now applying the conjectural bound in \eqref{4expect} yields 
$$
|\Sigma|^8\ll Q^{15/2+\varepsilon}
$$
upon recalling that $r=Q^{3/2}$. 
This is non-trivial! In fact, any improvement by a factor of $R^{-\eta}$ with $\eta>0$ of the established estimate \eqref{establishedbound} for $R=r^{1/3}$ would imply a non-trivial bound for $\Sigma$ in our special situation when \eqref{specialsituation} and $r=Q^{3/2}$ hold.  Theorem \ref{E4theo} gives such an improvement if $R$ is small compared to $r^{1/12}$. This range of improvement is not sufficient to obtain an {\it unconditional} non-trivial bound for $\Sigma$.  Thus, so far we get only a {\it conditional} improvement under the Hypothesis \ref{H2}. Combining everything above, the final bound is 
$$
P(x)\ll Q^{7/16+\varepsilon}
$$ 
in our situation. (Recall that the trivial bound is $P(x)\ll Q^{1/2}$.) If  the equations in \eqref{specialsituation} hold and $r$ is smaller than $Q^{3/2}$, then the analytic term \eqref{analyticterm} oscillates, so removing it by partial summation comes with a cost. Proceeding along the same lines as above, it turns out that these costs are moderate enough to still get a non-trivial bound under the Hypothesis \ref{H2} if $Q= N^{1/3}$ and $Q^{1+\eta}\le r\le Q^{3/2}$, where $\eta>0$ is arbitrarily small but fixed. In the next subsection, we will work out the details of these calculations for large $r$'s and general $Q$, $N$ and $z$. \\ \\
{\bf Remark.} If one avoids Wolke's approach using Diophantine approximation by fractions $b/r$, then one needs to bound the number of Farey fractions $a/q^2$ with $q\sim Q$ in intervals of the form $[a_0/q_0^2-\Delta,a_0/q_0^2+\Delta]$, where $q_0\sim Q$ is fixed. One may approach this problem in the same way as above via modular square roots modulo $q_0^2$. However, calculations show that in this case, our hypotheses on additive energies just lead us to a bound of $\ll Q^{1/2}$ for the said number of Farey fractions if $Q^3=N$, and therefore, no progress has been made along these lines. This indicates that an initial Diophantine approximation by Farey fractions $b/r$ is indeed essential. It allows us to replace the fractions $a_0/q_0^2$ by fractions $b/r$ with much smaller denominators. 

\subsection{Conditional treatment of the general case} \label{1} In this subsection we determine the $r$-range in which we get a conditional improvement of the known bound \eqref{Pxsimpler} along the lines above for general parameters $z$, $Q$ and $N$ satisfying \eqref{QNcond} and \eqref{zrange}.
Recall our bound \eqref{afterPoisson} for $P(x)$. We begin by cutting off the summation over $l$ at $N^{\varepsilon}Qr/\delta$, obtaining 
\begin{equation} \label{Pxcutoff}
P(x)\ll 1+\frac{\delta}{Qr} \left|\sum\limits_{|l| \le N^{\varepsilon}Qr/\delta} \sum\limits_{m\in \mathbb{Z}}W\left(\frac{l\delta}{Qr}\right)\Phi_5\left(\frac{m}{Q^2rz}\right) e\left(-\frac{l\sqrt{m}}{r^{3/2}z^{1/2}}\right) e_r\left(l\sqrt{jm}\right)\right|, 
\end{equation}
where we take the rapid decay of $W$ into account. Further, since $\Phi_5$ has compact support $[C_0,C_1]$ for suitable constants $C_1>C_0>0$, the $m$-range can be restricted to $M_0\le m\le M_1$ with 
\begin{equation} \label{Midef}
M_0:=C_0Q^2rz \quad \mbox{and} \quad M_1:=C_1Q^2rz.
\end{equation}
Next, we remove the oscillatory weight 
$$
W\left(\frac{l\delta}{Qr}\right)\Phi_5\left(\frac{m}{Q^2rz}\right) e\left(-\frac{l\sqrt{m}}{r^{3/2}z^{1/2}}\right)
$$
using partial summation in $l$ and $m$, getting
\begin{equation} \label{prelimPx}
\begin{split}
P(x)\ll & 1+\frac{N^{\varepsilon}Q}{r} \cdot \max\limits_{L\le N^{\varepsilon}Qr/\delta} \max\limits_{M_0\le M\le M_1} \left|
\sum\limits_{|l| \le L}  \sum\limits_{M_0\le m\le M} e_r\left(l\sqrt{jm}\right)\right|\\
\ll &  1+\frac{N^{\varepsilon}Q}{r} \cdot \max\limits_{M\le CQ^2rz} \Sigma(M),
\end{split}
\end{equation}
where 
$$
\Sigma(M):=\sum\limits_{|l| \le N^{\varepsilon}Qr/\delta} \left|   \sum\limits_{M_0\le m\le M} e_r\left(l\sqrt{jm}\right)\right|.
$$
To minimize the $l$-range above, we take $\delta$ as large as possible, i.e. $\delta:=Q^2$, so that
$$
\Sigma(M):=\sum\limits_{|l| \le L_1} \left|   \sum\limits_{M_0\le m\le M} e_r\left(l\sqrt{jm}\right)\right|
$$
with 
\begin{equation} \label{L0def}
L_1:=\frac{N^{\varepsilon}r}{Q}.
\end{equation}
Now applying the Cauchy-Schwarz inequality, we have
\begin{equation*}
\begin{split}
\Sigma(M)^2\ll & (L_1+1) \sum\limits_{|l| \le L_1} \left|   \sum\limits_{M_0\le m\le M} e_r\left(l\sqrt{jm}\right)\right|^2\\ 
= & (L_1+1)\sum\limits_{|l| \le L_1}  \sum\limits_{M_0\le m_1,m_2\le M} e\left(l\cdot \frac{\sqrt{jm_1}-\sqrt{jm_2}}{r}\right).
\end{split}
\end{equation*}
By convention, we take the modular square root of $jm$ always in such a way that $0\le \sqrt{jm}/r<1$. To bound the double sum above, we now apply the double large sieve, Proposition \ref{doublelargesieve}. Here we let the $\alpha_k$'s run over all integers with modulus not exceeding $L_1$ and the $\beta_l$'s over all real numbers $(\sqrt{jm_1}-\sqrt{jm_2})/r$ arising from pairs of integers $(m_1,m_2)$ with $M_0\le m_1,m_2\le M$. Taking $A:=L_1$ and $B:=1$ and setting all coefficients equal to 1, Proposition \ref{doublelargesieve} yields
$$
\Sigma(M)^2\ll \left(L_1+1\right)^2\cdot  \bigg(\sum\limits_{\substack{M_0\le m_1,m_2,m_3,m_4\le M\\ |(\sqrt{jm_1}-\sqrt{jm_2})-(\sqrt{jm_3}-\sqrt{jm_4})|/r\le 1/L_1}} 1 \bigg)^{1/2}.
$$
If 
\begin{equation} \label{rcond1}
r\ge Q,
\end{equation} 
this implies
$$
\Sigma(M)^4\ll L_1^4 \sum\limits_{\substack{1\le m_1,m_2,m_3,m_4\le M\\ |(\sqrt{jm_1}-\sqrt{jm_2})-(\sqrt{jm_3}-\sqrt{jm_4})|\le Q}} 1.
$$
Rewriting the right-hand side above, it follows that
\begin{equation*}
|\Sigma(M)|^4\ll L_1^4\sum\limits_{\substack{\lambda\in \mathbb{Z}\\ |\lambda|\le Q}} \sum\limits_{\substack{1\le m_1,m_2,m_3,m_4\le M\\ (\sqrt{jm_1}-\sqrt{jm_2})-(\sqrt{jm_3}-\sqrt{jm_4})\equiv \lambda \bmod{r}}} 1.
\end{equation*}
Applying Cauchy-Schwarz again, we deduce that
\begin{equation*}
\begin{split}
|\Sigma(M)|^8\ll & L_1^8Q  \sum\limits_{\lambda\bmod{r}} \bigg| \sum\limits_{\substack{1\le m_1,m_2,m_3,m_4\le M\\ (\sqrt{jm_1}-\sqrt{jm_2})-(\sqrt{jm_3}-\sqrt{jm_4})\equiv \lambda \bmod{r}}} 1\bigg|^2\\
= & L_1^8Q  \sum\limits_{\substack{1\le m_1,m_2,m_3,m_4,m_5,m_6,m_7,m_8\le M\\ (\sqrt{jm_1}-\sqrt{jm_2})-(\sqrt{jm_3}-\sqrt{jm_4})\equiv (\sqrt{jm_5}-\sqrt{jm_6})-(\sqrt{jm_7}-\sqrt{jm_8}) \bmod{r}}} \\
\ll & L_1^8Q  E_4(M;j,r),
\end{split}
\end{equation*}
where $E_4\left(M;j,r\right)$ is defined as in \eqref{AE4}. Now Hypothesis \ref{H2} gives
\begin{equation*}
|\Sigma(M)|^8\ll  L_1^8Q \left(\frac{M^8}{r}+M^4\right)N^{\varepsilon}
\end{equation*}
and hence
\begin{equation} \label{Sigma8est}
|\Sigma(M)|\ll \frac{N^{2\varepsilon}r}{Q^{7/8}}\left(\frac{Q^2\Delta^{1/2}}{r^{1/8}}+Q\Delta^{1/4}\right)
\end{equation}
using \eqref{L0def}, \eqref{Midef} and $rz\le \Delta^{1/2}$. Combining \eqref{prelimPx} and \eqref{Sigma8est}, we obtain
\begin{equation} \label{Pxlarge}
P(x)\ll \left(1+Q^{1/8}\left(\frac{Q^{2}\Delta^{1/2}}{r^{1/8}}+Q\Delta^{1/4}\right)\right)N^{\varepsilon}
\end{equation}
upon redefining $\varepsilon$. This is stronger than \eqref{Pxsimpler} if 
$$
N\le Q^{7/2}N^{-\eta}
$$
and  
\begin{equation} \label{thefirstrcond}
r\ge QN^{\eta}
\end{equation}
for some $\eta>0$. Note that \eqref{thefirstrcond}
is more restrictive than \eqref{rcond1}.
Recalling that $r\le N^{1/2}$, we may summarize the results in this subsection by saying that we beat the bound \eqref{Pxsimpler} if 
\begin{equation} \label{largercond}
Q^2\le N\le Q^{7/2}N^{-\eta} \quad \mbox{and} \quad  QN^{\eta}\le r\le N^{1/2}.
\end{equation}

\section{Handling medium-sized $r$'s}
\subsection{Sketch for our particular situation}
Let us return to our special situation when the equations in \eqref{specialsituation} hold. As seen in the previous subsection, if $r$ is smaller than $QN^{\eta}$, then the method above does not give a saving. The reason is that the oscillations of the analytic term become too large. It is therefore reasonable to remove these oscillations using Weyl differencing and then proceed in a similar way as before. Let us here focus on the case when $r=Q$. Then \eqref{specialPx} turns into 
\begin{equation} \label{specialPxmedium}
P(x)\ll 1+\frac{\delta}{Q^2} \left|\sum\limits_{l\in \mathbb{Z}} W\left(\frac{l\delta}{Q^2}\right)\sum\limits_{m\in \mathbb{Z}}\Phi_5\left(\frac{m}{Q^{1/2}}\right) e\left(-\frac{l\sqrt{m}}{Q^{1/4}}\right) e_r\left(l\sqrt{jm}\right)\right|. 
\end{equation}
The right-hand side behaves essentially like 
$$
E=1+\frac{\delta}{Q^2} \left|\sum\limits_{|l|\ll Q^2/\delta} \sum\limits_{m\asymp Q^{1/2}} e\left(-\frac{l\sqrt{m}}{Q^{1/4}}\right) e_r\left(l\sqrt{jm}\right)\right|. 
$$
Our treatment begins with an application of the Cauchy-Schwarz inequality, obtaining
$$
|E|^2\ll 1+\frac{\delta}{Q^2} \sum\limits_{|l|\ll Q^2/\delta}  \left| \sum\limits_{m\asymp Q^{1/2}} e\left(-\frac{l\sqrt{m}}{Q^{1/4}}\right) e_r\left(l\sqrt{jm}\right)\right|^2.
$$
Next, we apply Weyl differencing, Proposition \ref{Weyldif}, to deduce that
$$
|E|^2\ll \frac{Q}{H}+\frac{\delta}{HQ^{3/2}}\sum\limits_{1\le h\le H}\sum\limits_{|l|\ll Q^2/\delta} |\Sigma(h,l)|,
$$
where $H\ll Q^{1/2}$ is a suitable parameter which will be fixed later, and 
\begin{equation} \label{Sigmalh}
\Sigma(h,l):=\sum\limits_{m\asymp Q^{1/2}} e\left(-\frac{l(\sqrt{m+h}-\sqrt{m})}{Q^{1/4}}\right) e_r\left(l(\sqrt{j(m+h)}-\sqrt{jm})\right).
\end{equation}
(To be precise, $m$ runs over some interval $I_h$ depending on $h$ with endpoints in a range of $\asymp Q^{1/2}$.) Our target is to derive a non-trivial estimate of the form
$$
\sum\limits_{|l|\ll Q^2/\delta} |\Sigma(h,l)| \ll \frac{Q^{5/2-\eta}}{\delta}
$$
for some $\eta>0$.
We observe that the amplitude in the analytic term  on the right-hand side of \eqref{Sigmalh} is bounded by
$$
\frac{l(\sqrt{m+h}-\sqrt{m})}{Q^{1/4}}\ll \frac{HQ^{3/2}}{\delta}.
$$ 
It is therefore advisable to choose $\delta$ and $H$ in such a way that 
\begin{equation} \label{Hdelta}
 \frac{HQ^{3/2}}{\delta}\asymp 1,
\end{equation}
so that the analytic term does not oscillate. (Of course, we need to make sure that this consists with the conditions on $\delta$ in \eqref{deltarange}.) Under this choice, our bound for $|E|^2$ simplifies into
 $$
|E|^2\ll \frac{Q}{H}+\sum\limits_{1\le h\le H}\sum\limits_{|l|\ll Q^{1/2}/H} |\Sigma(h,l)|,
$$
and $\Sigma(h,l)$ can essentially be replaced by 
$$
\Sigma'(h,l):=\sum\limits_{m\asymp Q^{1/2}}  e_r\left(l(\sqrt{j(m+h)}-\sqrt{jm})\right).
$$
Now our target becomes to derive a non-trivial estimate of the form
$$
\sum\limits_{1\le h\le H}
\sum\limits_{|l|\ll Q^{1/2}/H} |\Sigma'(h,l)| \ll Q^{1-\eta}.
$$
Applying the double large sieve in a similar fashion as in the previous section, we have
\begin{equation*}
\begin{split}
\left(\sum\limits_{|l|\ll Q^{1/2}/H} |\Sigma'(h,l)|\right)^2 =& \left(\sum\limits_{|l|\ll Q^{1/2}/H} \epsilon_{h,l}\Sigma'(h,l)\right)^{2}\\
\ll & \frac{Q}{H^2} \sum\limits_{\substack{m_1,m_2\asymp Q^{1/2}\\ ||(f_{j,h}(m_1)-f_{j,h}(m_2))/r||\ll H/Q^{1/2}}}  1,
\end{split}
\end{equation*}
where $\epsilon_{h,l}$ are suitable complex number of modulus 1 and 
\begin{equation} \label{fhdef}
f_{j,h}(jm):=\sqrt{j(m+h)}-\sqrt{jm}. 
\end{equation}
Another application of Cauchy-Schwarz gives
\begin{equation*}
\begin{split}
\bigg(\sum\limits_{\substack{m_1,m_2\asymp Q^{1/2}\\ ||(f_{j,h}(m_1)-f_{j,h}(m_2))/r||\le H/Q^{1/2}}}  1\bigg)^2 =& \bigg(\sum\limits_{|\lambda|\ll rH/Q^{1/2}} \sum\limits_{\substack{m_1,m_2\asymp Q^{1/2}\\ f_{j,h}(m_1)-f_{j,h}(m_2)\equiv \lambda\bmod{r}}}  1\bigg)^2\\
\ll &\frac{rH}{Q^{1/2}}\cdot  \sum\limits_{\lambda\bmod r} \bigg( \sum\limits_{\substack{m_1,m_2\asymp Q^{1/2}\\ f_{j,h}(m_1)-f_{j,h}(m_2)\equiv \lambda\bmod{r}}}  1\bigg)^2\\
\ll & \frac{rH}{Q^{1/2}}\cdot  F_2(cQ^{1/2};j,h,r)\\ = & HQ^{1/2} F_2(cQ^{1/2};j,h,r),
\end{split}
\end{equation*}
where $c>0$ is a suitable constant and $F_2(cQ^{1/2};j,h,r)$ is the additive energy defined in \eqref{F2def}. Under  Hypothesis \ref{H3}, this implies 
\begin{equation*} 
\begin{split}
\sum\limits_{\substack{m_1,m_2\asymp Q^{1/2}\\ ||(f_{j,h}(m_1)-f_{j,h}(m_2))/r||\le H/Q^{1/2}}}  1\ll
 H^{1/2}Q^{1/4} \left((h,r)\cdot \frac{Q^2}{r}+Q\right)^{1/2}N^{\varepsilon}\ll H^{1/2}Q^{3/4}N^{\varepsilon}(h,r)^{1/2},
\end{split}
\end{equation*}
and hence, using Lemma \ref{gcdsums}
$$
\sum\limits_{1\le h\le H}
\sum\limits_{|l|\ll Q^{1/2}/H} |\Sigma'(h,l)| \ll \frac{N^{\varepsilon}Q^{7/8}}{H^{3/4}}\cdot \sum\limits_{1\le h\le H} (h,r)^{1/2}\ll Q^{7/8}H^{1/4}(rN)^{\varepsilon},
$$
which is non-trivial if $H$ is small. Altogether, upon redefining $\varepsilon$, we arrive at an estimate of 
$$
P(x)^2\ll \frac{Q}{H}+H^{1/4}Q^{7/8}N^{\varepsilon}. 
$$
The optimal choice is $H:=Q^{1/10}$, giving
$$
P(x)\ll Q^{9/20}N^{\varepsilon}. 
$$
Recalling \eqref{Hdelta}, we take
$$
\delta:=Q^{8/5}.
$$
This is consistent with \eqref{deltarange} because in our situation,
$$
\frac{Q^2\Delta}{z}=\frac{Q^2\Delta}{\Delta^{1/2}/r}=Q^3\Delta^{1/2}=Q^{3/2}< Q^{8/5}=\delta< Q^2.
$$
In the next subsection, we will work out the details of the above calculations for medium-sized $r$'s and general $Q$, $N$ and $z$. It will turn out that if the equations in \eqref{specialsituation} hold, we obtain a non-trivial bound for $P(x)$ if $Q^{1/2+\eta}\le r\le Q^{9/8}$ along these lines.

\subsection{Conditional treatment of the general case} \label{mediumgen} Applying the Cauchy-Schwarz inequality to \eqref{Pxcutoff}, and recalling that $m$ can be restricted to the range $M_0\le m\le M_1$ with $M_0$, $M_1$ as given in \eqref{Midef}, we obtain
\begin{equation*}
P(x)^2\ll 1+\frac{N^{2\varepsilon}}{L_2} \sum\limits_{|l| \le L_2} \left|\sum\limits_{M_0\le m\le M_1} \Phi_5\left(\frac{m}{Q^2rz}\right) e\left(-\frac{l\sqrt{m}}{r^{3/2}z^{1/2}}\right) e_r\left(l\sqrt{jm}\right)\right|^2,
\end{equation*} 
where 
\begin{equation} \label{L2def}
L_2:=\frac{N^{\varepsilon}Qr}{\delta}.
\end{equation}
If 
\begin{equation} \label{rangecond}
L_2\ge 2 \quad \mbox{and} \quad H\le M_1-M_0,
\end{equation}
then using Weyl differencing, Proposition \ref{Weyldif}, we deduce that
\begin{equation} \label{afterWeyl}
P(x)^2\ll \frac{N^{2\varepsilon} Q^4r^2z^2}{H} +\frac{N^{2\varepsilon}Q^2rz}{HL_2} \cdot \sum\limits_{1\le h\le H} 
  \sum\limits_{|l| \le L_2} |\Sigma(h,l)|
\end{equation}
with  
\begin{equation} \label{Sigmahldef}
\Sigma(h,l):=
 \sum\limits_{m\in \mathbb{Z}} V_h\left(\frac{m}{Q^2rz}\right)e\left(-\frac{l(\sqrt{m+h}-\sqrt{m})}{r^{3/2}z^{1/2}}\right) e_r\left(l(\sqrt{j(m+h)}-\sqrt{jm})\right),
\end{equation}
where 
$$
V_h(y):=\Phi_5(y)\Phi_5\left(y+\frac{h}{Q^2rz}\right).
$$
We note that like $\Phi_5$, this function $V_h$ has compact support in $[C_0,C_1]\subset \mathbb{R}_{>0}$.  
We remove the weight
$$
 V_h\left(\frac{m}{Q^2rz}\right)e\left(-\frac{l(\sqrt{m+h}-\sqrt{m})}{r^{3/2}z^{1/2}}\right)
$$
using partial summation in $l$ and $m$, getting
\begin{equation*} 
\sum\limits_{|l| \le L_2} |\Sigma(h,l)| \ll \left(1+\frac{HL_2}{Qr^2z}\right)\cdot  \max\limits_{M_0\le M\le M_1} \Sigma_h(M),
\end{equation*}
where
\begin{equation*}
\Sigma_h(M)= \sum\limits_{|l|\le L_2} \left| \sum\limits_{M_0\le m\le M} e_r\left(l(\sqrt{j(m+h)}-\sqrt{jm})\right)\right|. 
\end{equation*}
We shall choose $\delta$ and $H$ in such a way that 
\begin{equation} \label{L2choice}
\frac{HL_2}{Qr^2z}= 1,
\end{equation}
getting
\begin{equation} \label{remove}
\sum\limits_{|l| \le L_2} |\Sigma(h,l)| \ll \max\limits_{M_0\le M\le M_1} \Sigma_h(M).
\end{equation}
Applying the double large sieve, Proposition \ref{doublelargesieve}, in a similar way as in the previous section, we obtain
\begin{equation} \label{newmark}
|\Sigma_h(M)|^2\ll L_2^2 \sum\limits_{\substack{M_0\le m_1,m_2\le M\\ ||(f_{j,h}(m_1)-f_{j,h}(m_2))/r||\le 1/L_2}} 1, 
\end{equation}
where $f_{j,h}(m)$ is defined as in \eqref{fhdef}. If
\begin{equation} \label{L2condi}
L_2\le r,
\end{equation}
then another application of Cauchy-Schwarz now gives
\begin{equation*}
\begin{split}
|\Sigma_h(M)|^4
=& L_2^4\bigg(\sum\limits_{|\lambda|\le r/L_2} \sum\limits_{\substack{M_0\le m_1,m_2\le M\\ f_{j,h}(m_1)-f_{j,h}(m_2)\equiv \lambda\bmod{r}}}  1\bigg)^2\\
\le &rL_2^3 \sum\limits_{\lambda\bmod r} \bigg( \sum\limits_{\substack{M_0\le m_1,m_2\le M\\ f_{j,h}(m_1)-f_{j,h}(m_2)\equiv \lambda\bmod{r}}}  1\bigg)^2\\
\le & rL_2^3 F_2(M;j,h,r),
\end{split}
\end{equation*}
where the additive energy $F_2(M;j,h,r)$ is defined as in \eqref{F2def}. Under Hypothesis \ref{H3}, a bound of the form
\begin{equation} \label{afterH3}
|\Sigma_h(M)|^4\ll N^{\varepsilon}rL_2^3 \left((h,r)\cdot \frac{M^4}{r}+M^2\right)
\end{equation}
follows. Combining \eqref{afterWeyl}, \eqref{remove} and \eqref{afterH3}, and using Lemma \ref{gcdsums}, we arrive at the bound
\begin{equation*}
\begin{split}
P(x)^2\ll & \frac{N^{2\varepsilon} Q^4r^2z^2}{H} +\frac{N^{2\varepsilon}Q^2rz}{HL_2} \cdot
 N^{\varepsilon/4}r^{1/4}L_2^{3/4} \sum\limits_{1\le h\le H} \left((h,r)\cdot \frac{M_1}{r^{1/4}}+M_1^{1/2}\right)\\
\ll & \frac{N^{2\varepsilon} Q^4r^2z^2}{H} +\frac{N^{3\varepsilon}Q^2rz}{L_2^{1/4}}\cdot  
 \left(M_1+M_1^{1/2}r^{1/4}\right),
\end{split}
\end{equation*}
subject to the conditions in \eqref{rangecond}, \eqref{L2choice} and \eqref{L2condi}. Since
\begin{equation} \label{L2setting}
L_2=\frac{Qr^2z}{H}
\end{equation}
from \eqref{L2choice}, we thus obtain
\begin{equation*}
\begin{split}
P(x)^2\ll & \frac{N^{2\varepsilon} (Q^2rz)^2}{H} +N^{3\varepsilon}Q^{7/4}r^{1/2}z^{3/4}H^{1/4}
 \left(M_1+M_1^{1/2}r^{1/4}\right)\\ \ll & N^{3\varepsilon}\left(\frac{(Q^2rz)^2}{H} +Q^{15/4}r^{3/2}z^{7/4}H^{1/4}+Q^{11/4}r^{5/4}z^{5/4}H^{1/4}\right),
\end{split}
\end{equation*}
where for the second line, we have used \eqref{Midef}. Now we choose $H$ in such a way that the first and second terms above are balanced, i.e.
\begin{equation} \label{H82}
H:=Q^{1/5}r^{2/5}z^{1/5},
\end{equation}
yielding
\begin{equation*}
P(x)^2\ll N^{3\varepsilon}\left(Q^{19/5}r^{8/5}z^{9/5}+Q^{14/5}r^{27/20}z^{13/10}\right)
\end{equation*}
and so 
\begin{equation} \label{Pxmediumbound}
P(x)\ll N^{2\varepsilon}\left(Q^{19/10}\Delta^{9/20}r^{-1/10}+Q^{7/5}\Delta^{13/40}r^{1/40}\right)
\end{equation}
using $rz\le \Delta^{1/2}$ and taking square root. The choice of $H$ in \eqref{H82} together with \eqref{L2choice} give
\begin{equation} \label{L82}
L_2=Q^{4/5}r^{8/5}z^{4/5}.
\end{equation}
Recalling \eqref{Midef} and \eqref{L2setting}, the above choice of $H$ is consistent with the conditions in \eqref{rangecond} if 
\begin{equation} \label{rz82}
rz\ge D\max\left\{Q^{-9/4}r^{1/4},Q^{-1}r^{-1}\right\}
\end{equation}
for a suitable constant $D>0$. Also, \eqref{L2condi} holds if 
$$
L_2=Q^{4/5}r^{8/5}z^{4/5}\le r.
$$
Using $rz\le \Delta^{1/2}$ and $\Delta\le Q^{-2}$, it is readily checked that this holds for all $r\ge 1$.
It remains to check that \eqref{L2setting} is consistent with \eqref{deltarange}, keeping in mind that the relation between $L_2$ and $\delta$ is given by \eqref{L2def}. So we have 
$$
\delta=\frac{N^{\varepsilon}Qr}{L_2},
$$
and hence \eqref{deltarange} translates into the condition
$$
\frac{N^{\varepsilon}r}{Q}\le L_2=Q^{4/5}r^{8/5}z^{4/5}\le \frac{N^{\varepsilon}rz}{Q\Delta}.
$$
This holds iff 
\begin{equation} \label{rzcondi}
rz\ge \max\left\{Q^{9}\Delta^5r^4, N^{5\varepsilon/4}r^{1/4}Q^{-9/4}\right\}. 
\end{equation}
Combining this with \eqref{rz82}, we get a condition of 
\begin{equation} \label{rzmediumsizecondition}
rz\ge \max\left\{Q^{9}\Delta^5r^4, N^{5\varepsilon/4}r^{1/4}Q^{-9/4},Q^{-1}r^{-1}\right\},
\end{equation}
which we keep in mind for our later calculation of the final bound for $P(x)$. In the remainder of this subsection, we check in which $r$-range we get a non-trivial bound for $P(x)$ in our special situation when the equations in \eqref{specialsituation} hold. In this case, \eqref{rzmediumsizecondition} turns into 
$$
Q^{1/2}\le r\le Q^{9/8}. 
$$
Moreover, \eqref{Pxmediumbound} is non-trivial (i.e., it gives a bound for $P(x)$ which is much smaller than $\ll Q^{1/2}$) if $r\ge Q^{1/2+30\varepsilon}$. So we obtain a conditional improvement if 
$$
Q^{1/2+\eta}\le r\le Q^{9/8}
$$ 
for some $\eta>0$ in this situation. 

\section{Handling small $r$'s}
\subsection{Sketch for our particular situation} Again, we return to the situation given in \eqref{specialsituation}. We may view the $m$-sum in \eqref{afterWeyl} as an incomplete weighted exponential sum modulo $r$. If the $m$-range is much larger than the square root $r^{1/2}$ of the modulus, then completing these exponential sums potentially gives a saving. A standard way to do this is to break the $m$-summation into residue classes modulo $r$ and use the Poisson summation formula. In this subsection, we want to look at the particular case when $r=Q^{1/2}$. Recall that the $m$-range is of size $\asymp Q^2rz=Q^2\Delta^{1/2}=Q^{1/2}$ if the equations in \eqref{specialsituation} hold. So in this case, the $m$-sum is nearly a complete exponential sum.  Under the above conditions, \eqref{afterWeyl} turns into
\begin{equation*}
P(x)^2\ll \frac{N^{2\varepsilon} Q}{H} +\frac{N^{2\varepsilon}Q^{1/2}}{HL_2} \cdot \sum\limits_{1\le h\le H} 
  \sum\limits_{|l| \le L_2} |\Sigma(h,l)|
\end{equation*}
with  
\begin{equation} \label{L2reminder}
L_2=\frac{N^{\varepsilon}Q^{3/2}}{\delta}\ge 2, \quad 1\le H\ll Q^{1/2}
\end{equation}
and 
\begin{equation*}
\Sigma(h,l):=
 \sum\limits_{m\in \mathbb{Z}} V_h\left(\frac{m}{Q^{1/2}}\right)e\left(-l(\sqrt{m+h}-\sqrt{m})Q^{1/4}\right) e_r\left(l(\sqrt{j(m+h)}-\sqrt{jm})\right).
\end{equation*}
Breaking the $m$-summation into residue classes modulo $r$ gives
\begin{equation} \label{Sigmanew}
\Sigma(h,l)= \sum\limits_{a=1}^r e_r\left(l(\sqrt{j(a+h)}-\sqrt{ja})\right) \cdot
 \sum\limits_{m\equiv a\bmod{r}} V_h\left(\frac{m}{Q^{1/2}}\right)e\left(-l(\sqrt{m+h}-\sqrt{m})Q^{1/4}\right). 
\end{equation}
We set 
$$
W_h(y):=V_h(y)e\left(-l\left(\sqrt{y+hQ^{-1/2}}-\sqrt{y}\right)Q^{1/2}\right)
$$
to write the $m$-sum above as 
\begin{equation} \label{Whnew}
\sum\limits_{m\equiv a\bmod{r}} V_h\left(\frac{m}{Q^{1/2}}\right)e\left(-l(\sqrt{m+h}-\sqrt{m})Q^{1/4}\right)=
\sum\limits_{m\equiv a\bmod{r}} W_h\left(\frac{m}{Q^{1/2}}\right).
\end{equation}
Now Proposition \ref{Poisum} with $L=Q^{1/2}=r$ and $M=0$ implies that 
\begin{equation} \label{Poissonnew}
\sum\limits_{m\equiv a\bmod{r}} W_h\left(\frac{m}{Q^{1/2}}\right)=\frac{Q^{1/2}}{r} \sum\limits_{n\in \mathbb{Z}} \hat{W}_h\left(\frac{nQ^{1/2}}{r}\right) e_r\left(na\right)= \sum\limits_{n\in \mathbb{Z}} \hat{W}_h(n) e_r\left(na\right),
\end{equation}
where 
\begin{equation*}
\hat{W}_h(n)= \int\limits_{\mathbb{R}} W_h(y)e(-ny){\rm d}y
= 
\int\limits_{\mathbb{R}} V_h(y)e\left(-l\left(\sqrt{y+hQ^{-1/2}}-\sqrt{y}\right)Q^{1/2}-ny\right){\rm d}y.
\end{equation*}
Since $V_h$ has compact support in $[C_0,C_1]\subset \mathbb{R}_{>0}$ and $hQ^{-1/2}\ll 1$, we have
$$
\left|\frac{{\rm d}}{{\rm d}y} \left(l\left(\sqrt{y+hQ^{-1/2}}-\sqrt{y}\right)Q^{1/2}\right)\right|\asymp |lh|\ll L_2H
$$
if $y\in [C_0,C_1]$. Hence, using
$L_2H\ge 1$ and integration by parts, the Fourier integral above becomes negligible if $|n|\ge N^{\varepsilon}L_2H$. Otherwise, this Fourier integral is trivially bounded by $O(1)$. Taking this into account, combining \eqref{Sigmanew}, \eqref{Whnew} and \eqref{Poissonnew}, and  interchanging summations, we arrive at
\begin{equation*}
\Sigma(h,l)\ll \sum\limits_{n\le N^{\varepsilon}L_2H} |\mathcal{E}_{j,h}(l,n)|, 
\end{equation*}
where $\mathcal{E}_h(n)$ is the complete exponential sum
\begin{equation} \label{complete}
\mathcal{E}_{j,h}(l,n):=\sum\limits_{a=1}^r e_r\left(l(\sqrt{j(a+h)}-\sqrt{ja})+na\right).  
\end{equation}
It follows that
\begin{equation*}
P(x)^2\ll \frac{N^{2\varepsilon} Q}{H} +\frac{N^{2\varepsilon}Q^{1/2}}{HL_2} \cdot \sum\limits_{1\le h\le H} 
  \sum\limits_{|l| \le L_2}\sum\limits_{|n|\le N^{\varepsilon}L_2H} |\mathcal{E}_{j,h}(l,n)|.
\end{equation*}
In the appendix, we will prove the following bound.

\begin{Lemma}\label{expsumbound}
Let $r\in \mathbb{N}$ and $j,h,l,n\in \mathbb{Z}$, where $(j,r)=1$. Then
\begin{equation} \label{theexpsumbound}
\mathcal{E}_{j,h}(l,n)\ll r^{4/5+\varepsilon}(h,r)(l,r)^{1/5}.
\end{equation}
\end{Lemma}

Moreover, if $l=0$, we trivially have
\begin{equation} \label{thetrivialexpsumbound}
\mathcal{E}_{j,h}(0,n)=\begin{cases} r \mbox{ if } n\equiv 0\bmod{r}\\ 0 \mbox{ otherwise.} \end{cases}
\end{equation}
From \eqref{theexpsumbound} and \eqref{thetrivialexpsumbound}, we deduce that
\begin{equation*}
\begin{split}
P(x)^2\ll N^{2\varepsilon}\bigg(\frac{Q}{H} +& \frac{Q^{1/2}}{HL_2} \cdot \sum\limits_{1\le h\le H} 
  \sum\limits_{1\le |l| \le L_2}\sum\limits_{|n|\le N^{\varepsilon}L_2H} r^{4/5+\varepsilon}(h,r)(l,r)^{1/5}+\\
&  \frac{Q^{1/2}}{L_2} \cdot r\left(1+\frac{N^{\varepsilon}L_2H}{r}\right)\bigg)
\end{split}
\end{equation*}
which, using Lemma \ref{gcdsums}, implies the bound 
\begin{equation*}
P(x)^2\ll N^{10\varepsilon}\left(\frac{Q}{H} +\frac{Q^{1/2}}{L_2} \cdot \left(r+L_2^2Hr^{4/5}\right)\right)=
N^{10\varepsilon}\left(\frac{Q}{H} +\frac{Q}{L_2} +L_2HQ^{9/10}\right),
\end{equation*} 
where we recall that we assumed $r=Q^{1/2}$. 
We balance terms above, choosing  
\begin{equation} \label{newL2choice}
H_2:=Q^{1/30} \quad \mbox{and} \quad L_2:=Q^{1/30},
\end{equation}
thus obtaining the non-trivial bound
$$
P(x)\ll N^{5\varepsilon}Q^{29/60}
$$
after taking square root.
We still need to check that the above choice of $L_2$ is consistent with the conditions on $\delta$ in \eqref{deltarange}. From \eqref{L2reminder} and \eqref{newL2choice}, we have 
$$
\delta=N^{\varepsilon}Q^{22/15},
$$
and in our situation the $\delta$-range in \eqref{deltarange} turns into
$$
\frac{Q^2\Delta}{z}=\frac{Q^2\Delta}{\sqrt{\Delta}/r}=Q\le \delta\le Q^2.
$$
This is fine if $\varepsilon$ is small enough, and once again, we have obtained a non-trivially bound for $P(x)$ - this time at the point $r=Q^{1/2}$. Again, we will carry out the above method for general $Q$, $N$ and $z$ in the following subsection. In the situation of \eqref{specialsituation}, we will obtain a non-trivial bound for $P(x)$ if $Q^{\eta}\le r\le Q^{15/26-\eta}$ for some $\eta>0$. 

\subsection{Treatment of the general case} \label{3} We begin like in subsection \ref{mediumgen}. 
Breaking the $m$-sum on the right-hand side of \eqref{Sigmahldef}  into residue classes modulo $r$ gives
\begin{equation} \label{Sigmanewgen}
\Sigma(h,l)= \sum\limits_{a=1}^r e_r\left(l(\sqrt{j(a+h)}-\sqrt{ja})\right) \cdot
 \sum\limits_{m\equiv a\bmod{r}} V_h\left(\frac{m}{Q^2rz}\right)e\left(-\frac{l(\sqrt{m+h}-\sqrt{m})}{r^{3/2}z^{1/2}}\right).
\end{equation}
We set 
$$
W_h(y):=V_h(y)e\left(-l\left(\sqrt{y+\frac{h}{Q^2rz}}-\sqrt{y}\right)\cdot \frac{Q}{r}\right)
$$
to write the $m$-sum on the right-hand side of \eqref{Sigmanewgen} as
\begin{equation} \label{Whnewgen}
\sum\limits_{m\equiv a\bmod{r}} V_h\left(\frac{m}{Q^2rz}\right)e\left(-\frac{l(\sqrt{m+h}-\sqrt{m})}{r^{3/2}z^{1/2}}\right)=\sum\limits_{m\equiv a\bmod{r}} W_h\left(\frac{m}{Q^2rz}\right).
\end{equation}
Now Proposition \ref{Poisum} with $L=Q^2rz$ and $M=0$ implies that 
\begin{equation} \label{Poissonnewgen}
\sum\limits_{m\equiv a\bmod{r}} W_h\left(\frac{m}{Q^{1/2}}\right)=Q^2z \sum\limits_{n\in \mathbb{Z}} \hat{W}_h\left(nQ^2z\right) e_r\left(na\right),
\end{equation}
where 
\begin{equation*}
\hat{W}_h(nQ^2z)= \int\limits_{\mathbb{R}} W_h(y)e(-nQ^2zy){\rm d}y
= 
\int\limits_{\mathbb{R}} V_h(y)e\left(-l\left(\sqrt{y+\frac{h}{Q^2rz}}-\sqrt{y}\right)\cdot \frac{Q}{r}-nQ^2zy\right){\rm d}y.
\end{equation*}
Since $V_h$ has compact support in $[C_0,C_1]\subset \mathbb{R}_{>0}$ and $h/(Q^2rz)\ll 1$, we have
$$
\left|\frac{{\rm d}}{{\rm d} y} \left(l\left(\sqrt{y+\frac{h}{Q^2rz}}-\sqrt{y}\right)\cdot \frac{Q}{r}\right)\right|\asymp \frac{|lh|}{Q^2rz}\cdot \frac{Q}{r}\ll \frac{L_2H}{Qr^2z}
$$
if $y\in [C_0,C_1]$. Suppose that
\begin{equation} \label{newL2Hcond}
\frac{L_2H}{Qr^2z}\ge 1.
\end{equation}
Under this condition, using integration by parts, the Fourier integral above becomes negligible if $|n|\ge N^{\varepsilon}L_2H/(Q^3r^2z^2)$. For the remaining $n$'s, this Fourier integral is trivially bounded by $O(1)$. Taking these observations into account, combining \eqref{Sigmanewgen}, \eqref{Whnewgen} and \eqref{Poissonnewgen}, and  interchanging summations, we arrive at
\begin{equation} \label{Sigmaboundgensmall}
\Sigma(h,l)\ll Q^2z\sum\limits_{n\le N^{\varepsilon}L_2H/(Q^3r^2z^2)} |\mathcal{E}_{j,h}(l,n)|, 
\end{equation}
where $\mathcal{E}_h(n)$ is the complete exponential sum defined in \eqref{complete}.
Combining \eqref{afterWeyl}, \eqref{theexpsumbound} and  \eqref{Sigmaboundgensmall}, we obtain
\begin{equation*}
\begin{split}
P(x)^2\ll  \frac{N^{2\varepsilon} Q^4r^2z^2}{H} + & \frac{N^{2\varepsilon}Q^4rz^2}{HL_2}  \cdot \sum\limits_{1\le h\le H} 
  \sum\limits_{1\le |l| \le L_2}\sum\limits_{|n|\le N^{\varepsilon}L_2H/(Q^3r^2z^2)}r^{4/5+\varepsilon}(h,r)(l,r)^{1/5}+\\ 
&  \frac{N^{2\varepsilon}Q^4rz^2}{L_2}  \cdot r\left(1+\frac{N^{\varepsilon}L_2H}{Q^3r^3z^2}\right).
\end{split}
\end{equation*}
Using Lemma \ref{gcdsums}, 
this implies the bound 
\begin{equation*}
P(x)^2\ll N^{10\varepsilon}\left(\frac{Q^4r^2z^2}{H} +\frac{Q^4rz^2}{L_2}  \cdot \left(r+\frac{L_2^2H}{Q^3r^2z^2}\cdot r^{4/5}\right)\right)= N^{10\varepsilon}\left(\frac{Q^4r^2z^2}{H} +\frac{Q^4r^2z^2}{L_2}+\frac{QL_2H}{r^{1/5}}\right).
\end{equation*} 
We balance the terms above, choosing 
\begin{equation} \label{L2smallr}
H:=Qr^{11/15}z^{2/3} \quad \mbox{and} \quad L_2:=Qr^{11/15}z^{2/3},
\end{equation}
thus obtaining 
\begin{equation} \label{finalboundsmall}
P(x)\ll N^{5\varepsilon}Q^{3/2}r^{19/30}z^{2/3}\le N^{5\varepsilon}Q^{3/2}\Delta^{1/3}r^{-1/30}.
\end{equation}
This beats \eqref{Pxsimpler} if 
$$
N^{4\varepsilon}Q^{3/2}\Delta^{1/3}r^{-1/30}\le Q^2\Delta^{1/2},
$$
i.e.,
\begin{equation} \label{r1}
r\ge N^{48\varepsilon}Q^{-15}\Delta^{-5}. 
\end{equation}
We still need to check for which $r$-ranges the above choices of $H$ and $L_2$ are consistent with \eqref{deltarange}, \eqref{rangecond} and \eqref{newL2Hcond}. Using \eqref{L2smallr}, the inequalities in \eqref{rangecond} hold if
\begin{equation} \label{rz1}
rz\ge D\max\left\{Q^{-3/2}r^{-1/10},Q^{-3}r^{1/5}\right\}
\end{equation}
for a suitable constant $D>0$. 
Again using \eqref{L2smallr}, the inequality \eqref{newL2Hcond} holds if
\begin{equation} \label{rz3}
rz\ge Q^{-3}r^{13/5}. 
\end{equation}
From \eqref{L2def} and \eqref{L2smallr}, we infer that
$$
\delta=\frac{N^{\varepsilon}Qr}{L_2}=N^{\varepsilon}r^{4/15}z^{-2/3}.
$$
Hence, \eqref{deltarange} turns into 
\begin{equation*}
\frac{Q^2\Delta}{z}\le N^{\varepsilon}r^{4/15}z^{-2/3}\le Q^2. 
\end{equation*}
These inequalities are satisfied if
\begin{equation} \label{rz4}
rz\ge \max\left\{Q^6\Delta^3 r^{1/5},N^{2\varepsilon}Q^{-3}r^{7/5}\right\}.
\end{equation}
We record that \eqref{rz1}, \eqref{rz3}, \eqref{rz4} hold if 
\begin{equation} \label{rzsmallconds}
rz\ge N^{2\varepsilon}\max\left\{Q^{-3/2}r^{-1/10},Q^{-3}r^{13/5},Q^6\Delta^3 r^{1/5}\right\}.
\end{equation}
In our special situation when the equations in \eqref{specialsituation} are satisfied, this holds if
$$
Q^{\eta}\le r\le Q^{15/26-\eta}
$$ 
for a suitable $\eta>0$ depending on $\varepsilon>0$. Both parameters $\varepsilon$ and $\eta$ can be taken arbitrarily small. 

\section{Handling very small $r$'s}
\subsection{Sketch for our particular situation}
It remains to handle very small $r$'s. We look at the particular case when $r=1$ and the equations in \eqref{specialsituation} hold. Now we will {\it utilize} oscillations in the analytical term. We start from \eqref{Pxcutoff}, which in this case takes the form  
\begin{equation*} 
P(x)\ll 1+\frac{\delta}{Q} \left|\sum\limits_{|l| \le N^{\varepsilon}Q/\delta} \sum\limits_{m\in \mathbb{Z}}W\left(\frac{l\delta}{Q}\right)\Phi_5\left(\frac{m}{Q^{1/2}}\right) e\left(-l\sqrt{m}Q^{3/4}\right)\right|. 
\end{equation*}
We remove the weight functions using the triangle inequality and partial summation, getting
\begin{equation*} 
P(x)\ll 1+\frac{\delta}{Q} \sum\limits_{|l| \le N^{\varepsilon}Q/\delta} \max\limits_{M_0\le M\le M_1} \left| \sum\limits_{M_0\le m\le M} e\left(-l\sqrt{m}Q^{3/4}\right)\right| 
\end{equation*}
with $M_0$ and $M_1$ as in \eqref{Midef}, i.e. $M_i=C_iQ^{1/2}$ in our situation. For $M_0\le y\le M_1$ and $k\in \mathbb{N}$, we have
$$
\frac{{\rm d}^{k+2}}{{\rm d}y^{k+2}} \left(-l\sqrt{y}Q^{3/4}\right)\asymp |l|Q^{-k/2}.
$$
Now applying Proposition \ref{analyticexpsumbound} with $I:=[M_0,M]$, $f(y):=-l\sqrt{y}Q^{3/4}$ and $k:=1$, we have
$$
\sum\limits_{M_0\le m\le M} e\left(-l\sqrt{m}Q^{3/4}\right)\ll |l|^{1/6}Q^{5/12}\ll \frac{N^{\varepsilon}Q^{7/12}}{\delta^{1/6}}
$$
if $l\not=0$. It follows that
\begin{equation*} 
P(x)\ll 1+\frac{\delta}{Q^{1/2}}+\frac{N^{2\varepsilon}Q^{7/12}}{\delta^{1/6}}. 
\end{equation*}
We balance the second and third terms, choosing
$$
\delta:=Q^{13/14},
$$
which yields the non-trivial estimate
\begin{equation*} 
P(x)\ll Q^{3/7}. 
\end{equation*}
The above choice of $\delta$ consists with \eqref{deltarange}, which in our situation takes the form
$$
Q^{1/2}\le \delta\le Q^2.
$$
An analogous treatment of general $Q$, $N$ and $z$ in the following subsection will lead to a non-trivial bound for $P(x)$ in an $r$-range which in the case of \eqref{specialsituation} takes the form $r\le Q^{1/4-\eta}$ for some $\eta>0$. 

\subsection{Treatment of the general case} \label{4} Again, we start from \eqref{Pxcutoff}. Dividing the $m$-summation into residue classes modulo $r$ and using the triangle inequality and partical summation to remove the weight functions, we arrive at 
\begin{equation} \label{arrival}
P(x)\ll 1+\frac{\delta}{Qr} \sum\limits_{a=1}^r \sum\limits_{|l| \le N^{\varepsilon}Qr/\delta} \max\limits_{M_0\le M\le M_1} \left|\sum\limits_{\substack{M_0\le m\le M\\ m\equiv a\bmod{r}}} e\left(-\frac{l\sqrt{m}}{r^{3/2}z^{1/2}}\right) \right|,
\end{equation}
with $M_0$ and $M_1$ as in \eqref{Midef}. We rewrite the $m$-sum as 
$$
\sum\limits_{\substack{M_0\le m\le M\\ m\equiv a\bmod{r}}} e\left(-\frac{l\sqrt{m}}{r^{3/2}z^{1/2}}\right)=\sum\limits_{\substack{n\in \mathbb{Z}\\ M_0\le rn+a\le M}} e\left(-\frac{l\sqrt{rn+a}}{r^{3/2}z^{1/2}}\right)
$$
and note that 
$$
\frac{{\rm d}^{k+2}}{{\rm d}y^{k+2}} \left(-\frac{l\sqrt{ry+a}}{r^{3/2}z^{1/2}}\right)\asymp \frac{|l|}{r^{3/2}z^{1/2}}\cdot r^{k+2}\cdot (Q^2rz)^{-k-3/2}=\frac{|l|}{Q^{-2k-3}rz^{k+2}}
$$
if $M_0\le y\le M_1$. Now applying Proposition \ref{analyticexpsumbound} with 
$$
I:=\left[\frac{M_0-a}{r},\frac{M-a}{r}\right], \quad f(y):=\frac{-l\sqrt{ry+a}}{r^{3/2}z^{1/2}}\quad \mbox{and} \quad k:=1,
$$ 
we have
\begin{equation} \label{wehave}
\sum\limits_{\substack{n\in \mathbb{Z}\\ M_0\le rn+a\le M}} e\left(-\frac{l\sqrt{rn+a}}{r^{3/2}z^{1/2}}\right)\ll 
Q^2z\cdot \left(\frac{|l|}{Q^5rz^3}\right)^{1/6}+(Q^2z)^{3/4}+(Q^2z)^{1/4}\cdot \left(\frac{|l|}{Q^5rz^3}\right)^{-1/4}
\end{equation}
if $l\not=0$ and 
\begin{equation} \label{rverysmall1}
r\le M_1-M_0=(C_1-C_0)Q^2rz.
\end{equation}
Plugging the bound \eqref{wehave} into \eqref{arrival}, and estimating the contribution of $l=0$ trivially, we obtain
\begin{equation*}
\begin{split}
P(x)\ll & 1+\frac{\delta}{Qr} \cdot Q^2rz+ N^{2\varepsilon}r\left(Q^2z\cdot \left(\frac{1}{Q^4 z^3\delta}\right)^{1/6}+(Q^2z)^{3/4}+(Q^2z)^{1/4}\cdot \left(\frac{1}{Q^4z^3\delta}\right)^{-1/4}\right)\\
\ll &N^{2\varepsilon}\left(1+Qz\delta+\frac{Q^{4/3}rz^{1/2}}{\delta^{1/6}}+Q^{3/2}rz^{3/4}+Q^{3/2}rz\delta^{1/4}\right).
\end{split}
\end{equation*}
We balance the second and third terms in the last line above, choosing
$$
\delta:=Q^{2/7}r^{6/7}z^{-3/7},
$$
thus obtaining
\begin{equation} \label{Pxverysmallr}
P(x)\ll N^{2\varepsilon}\left(1+Q^{9/7}r^{6/7}z^{4/7}+Q^{3/2}rz^{3/4}+Q^{11/7}r^{17/14}z^{25/28}\right).
\end{equation}
Using $rz\le\Delta^{1/2}$, this implies 
\begin{equation} \label{Pxverysmallrconsequence}
P(x)\ll N^{2\varepsilon}\left(1+Q^{9/7}\Delta^{2/7}r^{2/7}+Q^{3/2}\Delta^{3/8}r^{1/4}+Q^{11/7}\Delta^{25/56}r^{9/28}\right).
\end{equation}
This beats \eqref{Pxsimpler} if 
\begin{equation} \label{rverysmall2}
r\le N^{-4\varepsilon}\min\left\{Q^{5/2}\Delta^{3/4},Q^2\Delta^{1/2},Q^{4/3}\Delta^{1/6}\right\}.
\end{equation}
We still need to check under which conditions the above choice of $\delta$ consists with \eqref{deltarange}. 
Under this choice, \eqref{deltarange} turns into
$$
\frac{Q^2\Delta}{z}\le Q^{2/7}r^{6/7}z^{-3/7}\le Q^2,
$$
which is the case if 
\begin{equation} \label{rzverysmall}
rz\ge \max\left\{Q^3\Delta^{7/4}r^{-1/2},Q^{-4}r^3\right\}.
\end{equation}
In the situation of \eqref{specialsituation}, reviewing \eqref{rverysmall1}, \eqref{rverysmall2} and \eqref{rzverysmall},
we obtain a non-trivial bound for $P(x)$ if $r\le Q^{1/4-\eta}$ for some $\eta>0$. 

\section{Proof of the main result}
In the previous sections, we obtained different estimates for $P(x)$ in different $r$-ranges. Some of these estimates come with a lower bound condition on $rz$. If $rz$ falls below this bound, we simply use the Lemma \ref{previousPxbound} to bound $P(x)$. This way, we obtain four different bounds which we then compare. The details are carried out below, where we redefine $\varepsilon$ suitably. 

In subsection \ref{1}, we obtained \eqref{Pxlarge} under the condition \eqref{rcond1}, which is
\begin{equation} \label{P1}
P(x)\ll \left(1+Q^{17/8}\Delta^{1/2}r^{-1/8}+Q^{9/8}\Delta^{1/4}\right)N^{\varepsilon} \quad \mbox{if } r\ge Q. 
\end{equation}
In subsection \ref{mediumgen}, we obtained \eqref{Pxmediumbound} under the condition \eqref{rzcondi}, which is 
\begin{equation} \label{P2}
\begin{split}
P(x)\ll & \left(Q^{19/10}\Delta^{9/20}r^{-1/10}+Q^{7/5}\Delta^{13/40}r^{1/40}\right)N^{\varepsilon} \quad \mbox{if } \\
rz\ge  & \max\left\{Q^{9}\Delta^5r^4, r^{1/4}Q^{-9/4},Q^{-1}r^{-1}\right\}N^{\varepsilon}. 
\end{split}
\end{equation}
In subsection \ref{3}, we obtained \eqref{finalboundsmall} under the conditions in  \eqref{rzsmallconds}, which is
\begin{equation} \label{P3}
\begin{split}
P(x)\ll & Q^{3/2}\Delta^{1/3}r^{-1/30}N^{\varepsilon} \quad \mbox{if }\\  
rz\ge & \max\left\{Q^{-3/2}r^{-1/10},Q^{-3}r^{13/5},Q^6\Delta^3 r^{1/5}\right\}N^{\varepsilon}.
\end{split}
\end{equation}
In subsection \ref{4}, we obtained \eqref{Pxverysmallrconsequence} under the conditions in \eqref{rzverysmall}, which is 
\begin{equation} \label{P4}
\begin{split}
P(x)\ll & \left(1+Q^{9/7}\Delta^{2/7}r^{2/7}+Q^{3/2}\Delta^{3/8}r^{1/4}+Q^{11/7}\Delta^{25/56}r^{9/28}\right)N^{\varepsilon} \quad \mbox{if }\\
rz\ge & \max\left\{Q^3\Delta^{7/4}r^{-1/2},Q^{-4}r^3\right\}.
\end{split}
\end{equation}
Using Lemma \ref{previousPxbound} if $rz$ falls below the lower bounds in \eqref{P2}, \eqref{P3} and \eqref{P4}, these results imply
\begin{equation} \label{P2'}
P(x)\ll  \left(Q^{19/10}\Delta^{9/20}r^{-1/10}+Q^{7/5}\Delta^{13/40}r^{1/40}+Q^{11}\Delta^5r^4+ r^{1/4}Q^{-1/4}+Qr^{-1}+1+Q^3\Delta\right)N^{\varepsilon},
\end{equation}
\begin{equation} \label{P3'}
P(x)\ll \left(Q^{3/2}\Delta^{1/3}r^{-1/30}+Q^{1/2}r^{-1/10}+Q^{-1}r^{13/5}+Q^8\Delta^3 r^{1/5}+1+Q^3\Delta\right)N^{\varepsilon}
\end{equation}
and 
\begin{equation} \label{P4'}
\begin{split}
& P(x)\ll \\
& \left(Q^{9/7}\Delta^{2/7}r^{2/7}+Q^{3/2}\Delta^{3/8}r^{1/4}+Q^{11/7}\Delta^{25/56}r^{9/28}+Q^5\Delta^{7/4}r^{-1/2}+Q^{-2}r^3+1+Q^3\Delta\right)N^{\varepsilon},
\end{split}
\end{equation}
respectively. We recall that we use \eqref{P1} for large, \eqref{P2'} for medium, \eqref{P3'} for small and \eqref{P4'} for very small $r$. To get the precise $r$-ranges, we compare large terms in \eqref{P1} and \eqref{P2'}, in \eqref{P2'} and \eqref{P3'}, and in \eqref{P3'} and \eqref{P4'} for these situations.   
We observe that 
\begin{equation} \label{range1}
Q^{11}\Delta^5r^4\le Q^{17/8}\Delta^{1/2}r^{-1/8}
\quad \mbox{if } r\le Q^{-71/33}\Delta^{-12/11},
\end{equation}
\begin{equation} \label{range2}
Q^{-1}r^{13/5}\le Q^{19/10}\Delta^{9/20}r^{-1/10} \quad \mbox{ if } r\le Q^{29/27}\Delta^{1/6}, 
\end{equation}
and 
\begin{equation} \label{range3}
Q^{9/7}\Delta^{2/7}r^{2/7}\le Q^{3/2}\Delta^{1/3}r^{-1/30} \quad \mbox{if } r\le Q^{45/67}\Delta^{10/67} 
\end{equation}
So we use
\begin{itemize}
\item \eqref{P1} if   $Q^{-71/33}\Delta^{-12/11}< r\le \Delta^{-1/2}$,
\item \eqref{P2'} if  $Q^{29/27}\Delta^{1/6}<r\le  Q^{-71/33}\Delta^{-12/11}$,
\item \eqref{P3'} if $Q^{45/67}\Delta^{10/67}<r\le Q^{29/27}\Delta^{1/6}$,
\item \eqref{P4'} if $r\le Q^{45/67}\Delta^{10/67}$
\end{itemize}
and note that 
$$
1< Q^{45/67}\Delta^{10/67}<Q^{29/27}\Delta^{1/6}<Q<Q^{-71/33}\Delta^{-12/11}\le \Delta^{-1/2}
$$
if 
$$
N^{39/142}<Q< N^{9/26},
$$
which we want to assume in the following. Hence, we record the following.

\begin{Proposition}\label{Propmain}  If $N^{39/142}<Q<N^{9/26}$, then under the Hypotheses \ref{H2} and \ref{H3}, we have 
\begin{equation*} 
\begin{split}
P(x)\ll & \left(1+Q^{17/8}\Delta^{1/2}r^{-1/8}+Q^{9/8}\Delta^{1/4}\right)N^{\varepsilon} \\ &\quad \mbox{\rm if } Q^{-71/33}\Delta^{-12/11}< r\le \Delta^{-1/2},\\
\ll & \left(Q^{19/10}\Delta^{9/20}r^{-1/10}+Q^{7/5}\Delta^{13/40}r^{1/40}+Q^{11}\Delta^5r^4+ r^{1/4}Q^{-1/4}+Qr^{-1}+1+Q^3\Delta\right)N^{\varepsilon} \\ &  \quad \mbox{\rm if }Q^{29/27}\Delta^{1/6}<r\le Q^{-71/33}\Delta^{-12/11},\\
\ll & \left(Q^{3/2}\Delta^{1/3}r^{-1/30}+Q^{1/2}r^{-1/10}+Q^{-1}r^{13/5}+Q^8\Delta^3 r^{1/5}+1+Q^3\Delta\right)N^{\varepsilon}\\ & \quad \mbox{\rm if }Q^{45/67}\Delta^{10/67}<r\le Q^{29/27}\Delta^{1/6},\\
\ll & \left(Q^{9/7}\Delta^{2/7}r^{2/7}+Q^{3/2}\Delta^{3/8}r^{1/4}+Q^{11/7}\Delta^{25/56}r^{9/28}+\right.\\ & \left. \quad Q^5\Delta^{7/4}r^{-1/2}+Q^{-2}r^3+1+Q^3\Delta\right)N^{\varepsilon} \\ & \quad \mbox{\rm if }r\le Q^{45/67}\Delta^{10/67}.
\end{split}
\end{equation*}
\end{Proposition}

Since we are particularly interested in the point $Q=N^{1/3}$, we want to work out which final bound we obtain in this case. If $Q=N^{1/3}$, then $\Delta=Q^{-3}$, and the above simplifies into
\begin{equation*} 
\begin{split}
P(x)\ll & \left(1+Q^{5/8}r^{-1/8}+Q^{3/8}\right)N^{\varepsilon} \quad \mbox{\rm if } Q^{37/33}< r\le Q^{3/2},\\
\ll & \left(Q^{11/20}r^{-1/10}+Q^{17/40}r^{1/40}+Q^{-4}r^4+Q^{-1/4}r^{1/4}+Qr^{-1}+1\right)N^{\varepsilon} \quad \mbox{\rm if }Q^{31/54}<r\le Q^{37/33},\\
\ll & \left(Q^{1/2}r^{-1/30}+Q^{-1}r^{13/5}+1\right)N^{\varepsilon}\quad \mbox{\rm if }Q^{15/67}<r\le Q^{31/54},\\
\ll & \left(Q^{3/7}r^{2/7}+Q^{3/8}r^{1/4}+Q^{13/56}r^{9/28}+Q^{-2}r^3+1\right)N^{\varepsilon} \quad \mbox{\rm if }r\le Q^{15/67}
\end{split}
\end{equation*}
and further into
\begin{equation*}
\begin{split}
P(x)\ll & Q^{16/33}N^{\varepsilon} \quad \mbox{if } 
Q^{37/33}< r\le Q^{3/2},\\
\ll & Q^{133/270}N^{\varepsilon}\quad \mbox{if } Q^{31/54}<r\le Q^{37/33},\\
\ll & Q^{133/270}N^{\varepsilon} \quad \mbox{if } Q^{15/67}<r\le Q^{31/54},\\
\ll & Q^{33/67}N^{\varepsilon} \quad \mbox{if } r\le Q^{15/67}.
\end{split}
\end{equation*}

So we have the following.

\begin{Corollary}
If $N=Q^3$, then under the Hypotheses \ref{H2} and \ref{H3}, we have $P(x)\ll Q^{133/270+\varepsilon}=Q^{1/2-1/135+\varepsilon}$. 
\end{Corollary}

This together with \eqref{Zhaobound} and \eqref{tildeDelta} gives our main result, Theorem \ref{mainresult}.

\section{An interesting connection to a short character sum} \label{interesting}
In addition to our considerations on the large sieve with square moduli, we will reveal an interesting connection between additive energies of modular square roots and certain short character sums with a cubic form in four variables. This connection comes by a completion argument using the Poisson summation formula. Employing the existing energy bounds gives a rather strong estimate for the said character sums if the summation ranges are about the square root of the modulus. Here we want to confine ourselves to the case when $r$ is an odd prime. 

Recalling \eqref{AE}, we may write
\begin{equation*}
E_2(R;j,r)= \sum\limits_{\substack{(k_1,k_2,k_3,k_4)\bmod r\\ k_1+k_2\equiv k_3+k_4\bmod{r}\\ \left\{\overline{j}k_i^2/r\right\}\le R/r \text{ for } i=1,2,3,4}} 1=\sum\limits_{\substack{(k_1,k_2,k_3,k_4)\bmod r\\ k_1+k_2\equiv k_3+k_4\bmod{r}}} \prod\limits_{i=1}^4 \chi_{(0,R/r]}\left(\left\{\frac{\overline{j}k_i^2}{r}\right\}\right),
\end{equation*}
where $\{y\}=y-[y]$ is the fractional part of $y\in \mathbb{R}$ and for an interval $I$, $\chi_I(t)$ is its characteristic function. We will smooth this characteristic function out, replacing it by  a 1-periodic function of the form
$$
\phi_{\nu}(y):=\sum\limits_{h\in \mathbb{Z}} \Phi\left(\frac{y+h}{\nu}\right), 
$$ 
where $\nu:=R/r$ and $\Phi$ is a Schwartz class function satisfying $0\le \Phi \le 1$  and supported in $[-1,1]$. We define the corresponding weighted additive energy as 
\begin{equation} \label{Etildedef}
\tilde{E}_2(R;j,r):=\sum\limits_{\substack{(k_1,k_2,k_3,k_4)\bmod r\\ k_1+k_2\equiv k_3+k_4\bmod{r}}} \prod\limits_{i=1}^4 \phi_{R/r}\left(\frac{\overline{j}k_i^2}{r}\right).
\end{equation}
We also define the additive energy 
\begin{equation*}
E_2'(R;j,r)= \sum\limits_{\substack{(k_1,k_2,k_3,k_4)\bmod r\\ k_1+k_2\equiv k_3+k_4\bmod{r}\\ \left|\left|k_i^2/r\right|\right|\le R/r \text{ for } i=1,2,3,4}} 1=\sum\limits_{\substack{(k_1,k_2,k_3,k_4)\bmod r\\ k_1+k_2\equiv k_3+k_4\bmod{r}}} \prod\limits_{i=1}^4 \chi_{[0,R/r]}\left(\left|\left|\frac{\overline{j}k_i^2}{r}\right|\right|\right).
\end{equation*}
The only differences between the definitions of $E_2(R;j,r)$ and $E_2'(R;j,r)$ are that in the definition of $E_2'(R;j,r)$, the fractional part is replaced by the distance to the nearest integer, and $k_i$'s congruent to $0$ modulo $r$ are no longer excluded. In the case of prime moduli $r$, the method of Kerr, Shkredov, Shparlinski and Zaharescu in \cite{KSSZ} gives the same bound for $E_2'(R;j,r)$  as the bound \eqref{E2result} for $E_2(R;j,r)$, namely
\begin{equation} \label{E2result2}
E_2'(R;j,r)\ll \left(\frac{R^{3/2}}{r^{1/2}}+1\right)R^{2+\varepsilon}. 
\end{equation}
(If $r$ is not a prime, then the exclusion of $k_i\equiv 0 \bmod{r}$ may matter, as we have seen in the discussion in section 2.) 
 We note that
\begin{equation} \label{E2tildeE2}
E_2(R;j,r)\le E_2'(R;j,r)\ll \tilde{E}_2(R;j,r)\ll \tilde{E}_2'(Rr^{\varepsilon};j,r),
\end{equation}
where the last inequality arises from the rapid decay of $\Phi$. 
Applying Poisson summation, Proposition \ref{Poisum} with the modulus 1, $L=\nu$ and $M=-y$, we have 
$$
\phi_{\nu}(y):=\nu\sum\limits_{h\in \mathbb{Z}} \hat{\Phi}\left(h\nu\right)e(hy).
$$
Plugging this into the right-hand side of \eqref{Etildedef}, and interchanging summations, we obtain
\begin{equation} \label{Etildedef1}
\tilde{E}_2(R;j,r)=\left(\frac{R}{r}\right)^4 \sum\limits_{h_1,h_2,h_3,h_4\in \mathbb{Z}}  \prod\limits_{i=1}^4 \hat{\Phi}\left(\frac{h_iR}{r}\right) \cdot S_4(j;h_1,h_2,h_3,h_4),
\end{equation}
where
$$
S_4(j;h_1,h_2,h_3,h_4):=\sum\limits_{\substack{(k_1,k_2,k_3,k_4)\bmod r\\ k_1+k_2\equiv k_3+k_4\bmod{r}}} e_r\left(\overline{j}\left(h_1k_1^2+h_2k_2^2+h_3k_3^3+h_4k_4^2\right)\right),
$$
which is a complete exponential sum in three variables ({\it three} variables because $k_1,k_2,k_3,k_4$ are interconnected by the relation $k_1+k_2\equiv k_3+k_4\bmod{r}$).  We will evaluate this exponential sum in the following.
Clearly, we can write
\begin{equation*}
S_4(j;h_1,h_2,h_3,h_4)=\sum\limits_{l=1}^r S_2(l;j;h_1,h_2)S_2(l;j;h_3,h_4),
\end{equation*}
where 
$$
S_2(l;j;a_1,a_2):=\sum\limits_{\substack{k_1,k_2\\ k_1+k_2\equiv l\bmod{r}}} e_r\left(\overline{j}\left(a_1k_1^2+a_2k_2^2\right)\right).
$$
We remove the variable $k_2$ above using $k_2\equiv l-k_1\bmod{r}$, thus getting a quadratic Gauss sum
$$
S_2(l;j;a_1,a_2)=\sum\limits_{k=1}^r e_r\left(\overline{j}\left((a_1+a_2)k^2-2a_2lk+a_2l^2\right)\right).
$$
We consider two cases below.\\ \\
{\bf Case 1:} $a_1+a_2\equiv 0 \bmod{r}$. Then
$$
S_2(l;j;a_1,a_2)=\begin{cases} 0 \mbox{  if } a_2l\not\equiv 0\bmod{r},\\ \\ r \mbox{ if } a_2l\equiv 0\bmod{r} \end{cases}
$$
by the orthogonality relation for additive characters. \\ \\
{\bf Case 2:} $a_1+a_2\not\equiv 0\bmod{r}$. Then using Proposition \ref{Gausssums} gives
\begin{equation} \label{secondcase}
S_2(l;j;a_1,a_2)=\epsilon_r\cdot \sqrt{r}\cdot \left(\frac{j}{r}\right)\cdot \left(\frac{a_1+a_2}{r}\right) \cdot 
e_r\left(\overline{j}a_2\left(1-\frac{a_2}{a_1+a_2}\right)l^2\right),
\end{equation}
where, for convenience, we have written 
$$
a_2\overline{(a_1+a_2)}=\frac{a_2}{a_1+a_2}.
$$ 
We can simplify the exponential term in Case 2 and combine both cases, getting
$$
S_2(l;j;a_1,a_2)=\epsilon_r\cdot \sqrt{r}\cdot \left(\frac{j}{r}\right)\cdot \left(\frac{a_1+a_2}{r}\right) \cdot 
e_r\left(\overline{j}l^2\cdot \frac{a_1a_2}{a_1+a_2}\right)+r\delta(a_1,a_2,l),
$$
where we set $a_1a_2/(a_1+a_2):=0$ if $a_1+a_2\equiv 0\bmod{r}$ and 
$$
\delta(a_1,a_2,l):=\begin{cases} 1 & \mbox{ if } a_1\equiv a_2\equiv 0\bmod{r} \mbox{ and } l\not\equiv 0\bmod{r},\\
1 & \mbox{ if } a_1\equiv -a_2\bmod{r} \mbox{ and } l\equiv 0\bmod{r},\\
0 & \mbox{ otherwise.}\end{cases}
$$
Hence, we get
\begin{equation*}
\begin{split}
S_4(j;h_1,h_2,h_3,h_4)= & \sum\limits_{l=1}^r\epsilon_r^2\cdot r \cdot \left(\frac{(h_1+h_2)(h_3+h_4)}{r}\right) \cdot 
e_r\left(\overline{j}l^2\left(\frac{h_1h_2}{h_3+h_4}+\frac{h_3h_4}{h_3+h_4}\right)\right)+\\
& \sum\limits_{l=1}^r\epsilon_r\cdot r^{3/2}\cdot \left(\frac{j}{r}\right)\cdot \left(\frac{h_1+h_2}{r}\right) \cdot 
e_r\left(\overline{j}l^2\cdot \frac{h_1h_2}{h_1+h_2}\right)\cdot \delta(h_3,h_4,l)+\\
& \sum\limits_{l=1}^r \epsilon_r\cdot r^{3/2}\cdot \left(\frac{j}{r}\right)\cdot \left(\frac{h_3+h_4}{r}\right) \cdot 
e_r\left(\overline{j}l^2\cdot \frac{h_3h_4}{h_3+h_4}\right)\cdot \delta(h_1,h_2,l)+\\
& \sum\limits_{l=1}^r r^2\cdot \delta(h_1,h_2,l)\delta(h_3,h_4,l). 
\end{split}
\end{equation*}
Again using Proposition \ref{Gausssums}, we have
$$
\sum\limits_{l=1}^r e_r\left(\overline{j}l^2\left(\frac{h_1h_2}{h_3+h_4}+\frac{h_3h_4}{h_3+h_4}\right)\right)=\epsilon_r\sqrt{r}\cdot \left(\frac{j}{r}\right)\left(\frac{h_1h_2/(h_1+h_2)+h_3h_4/(h_3+h_4)}{r}\right),
$$
$$
\sum\limits_{l=1}^r e_r\left(\overline{j}l^2\cdot \frac{h_1h_2}{h_1+h_2}\right)\delta(h_3,h_4,l)=\epsilon_r\sqrt{r}\cdot \left(\frac{j}{r}\right) \left(\frac{h_1h_2/(h_1+h_2)}{r}\right)\cdot \gamma(h_3,h_4)+\tilde{\gamma}(h_3,h_4),
$$
$$
\sum\limits_{l=1}^r e_r\left(\overline{j}l^2\cdot \frac{h_3h_4}{h_3+h_4}\right)\delta(h_1,h_2,l)=\epsilon_r\sqrt{r}\cdot \left(\frac{j}{r}\right) \left(\frac{h_3h_4/(h_3+h_4)}{r}\right)\cdot \gamma(h_1,h_2)+\tilde{\gamma}(h_1,h_2),
$$
$$
 \sum\limits_{l=1}^r\delta(h_1,h_2,l)\delta(h_3,h_4,l)=r\gamma(h_1,h_2)\gamma(h_3,h_4)+\tilde{\gamma}(h_1,h_2)\hat{\gamma}(h_3,h_4)+\hat{\gamma}(h_1,h_2)\tilde{\gamma}(h_3,h_4),
$$
where 
$$
\gamma(a_1,a_2):=\begin{cases} 1 & \mbox{ if } a_1\equiv a_2\equiv 0 \bmod{r},\\ 0 & \mbox{ otherwise},\end{cases}
$$
$$
\tilde{\gamma}(a_1,a_2):=\begin{cases} 1 & \mbox{ if } 0\not\equiv a_1\equiv -a_2 \bmod{r},\\ 0 & \mbox{ otherwise,}\end{cases}
$$
and 
$$
\hat{\gamma}(a_1,a_2):=\begin{cases} 1 & \mbox{ if } a_1\equiv -a_2 \bmod{r},\\ 0 & \mbox{ otherwise.}\end{cases}
$$
Putting the above together and simplifying gives
\begin{equation*}
\begin{split}
S_4(j;h_1,h_2,h_3,h_4)= & \epsilon_r^3\cdot r^{3/2} \cdot \left(\frac{j}{r}\right)\left(\frac{h_1h_2h_3+h_1h_2h_4+h_1h_3h_4+h_2h_3h_4}{r}\right) +\\
& \epsilon_r^2 \cdot r^{2}\cdot \left(\frac{h_1h_2}{r}\right) \cdot \gamma(h_3,h_4)+\\
& \epsilon_r\cdot r^{3/2}\cdot \left(\frac{j}{r}\right)\cdot \left(\frac{h_1+h_2}{r}\right) \cdot \tilde{\gamma}(h_3,h_4)+\\
& \epsilon_r^2 \cdot r^{2}\cdot \left(\frac{h_3h_4}{r}\right) \cdot \gamma(h_1,h_2)+\\
& \epsilon_r\cdot r^{3/2}\cdot \left(\frac{j}{r}\right)\cdot \left(\frac{h_3+h_4}{r}\right) \cdot \tilde{\gamma}(h_1,h_2)+\\
& r^3\gamma(h_1,h_2)\gamma(h_3,h_4)+r^2\tilde{\gamma}(h_1,h_2)\hat{\gamma}(h_3,h_4)+r^2\hat{\gamma}(h_1,h_2)\tilde{\gamma}(h_3,h_4).
\end{split}
\end{equation*} 
Plugging this into \eqref{Etildedef1}, we obtain
\begin{equation} \label{relation}
\begin{split}
\frac{\tilde{E}_2(R;j,r)}{(R/r)^4}=& r^3+\epsilon_r^3\cdot r^{3/2} \cdot \left(\frac{j}{r}\right)\cdot  \sum\limits_{h_1,h_2,h_3,h_4\in \mathbb{Z}}  \prod\limits_{i=1}^4\hat{\Phi}\left(\frac{h_iR}{r}\right)\cdot\left(\frac{h_1h_2h_3+h_1h_2h_4+h_1h_3h_4+h_2h_3h_4}{r}\right)\\ & + O\left(r^2\cdot \left(\frac{r}{R}\right)^2+r^{3/2}\cdot \left(\frac{r}{R}\right)^3\right)
\end{split}
\end{equation} 
if $R\le r$. Using \eqref{E2tildeE2} and estimating the right-hand side trivially, this gives the following new energy bound upon noting that $Rr^{1/2}\ll \left(R^2r^{3/2}\right)^{1/2}\ll R^2+r^{3/2}$.

\begin{Theorem} If $r$ is an odd prime, $j\in\mathbb{Z}$ with $(j,r)=1$ and $1\le R\le r$, then 
$$
E_2(R;j,r)\ll \frac{R^4}{r}+R^2+r^{3/2}.
$$
\end{Theorem}

This is superior over \eqref{E2result} if $R\ge r^{4/7}$, which is of moderate interest. More interesting is the converse direction in which we bound the above character sum using \eqref{E2result2} and \eqref{E2tildeE2}, getting the bound
\begin{equation*}
\begin{split}
& \sum\limits_{h_1,h_2,h_3,h_4\in \mathbb{Z}}  \prod\limits_{i=1}^4\hat{\Phi}\left(\frac{h_iR}{r}\right)\cdot  \left(\frac{h_1h_2h_3+h_1h_2h_4+h_1h_3h_4+h_2h_3h_4}{r}\right)\\
\ll & r^{-3/2}\cdot \left(\left(\frac{r}{R}\right)^4 \cdot \left(\frac{R^{3/2}}{r^{1/2}}+1\right)R^{2+\varepsilon}r^{\varepsilon}+r^{3}+
r^2\cdot \left(\frac{r}{R}\right)^2+r^{3/2}\cdot \left(\frac{r}{R}\right)^3\right)\\
\ll & \left(\frac{r^2}{R^{1/2}}+\frac{r^{5/2}}{R^2}+r^{3/2}+\frac{r^3}{R^3}\right)(Rr)^{\varepsilon}
\end{split}
\end{equation*}
if $1\le R\le r$. Writing $M:=r/R$, this implies the following result.

\begin{Theorem} If $r$ is an odd prime, $1\le M\le r$ and $W:\mathbb{R}\rightarrow \mathbb{C}$ is a Schwartz class function whose Fourier transform satisfies $0\le \hat{W}\le 1$ and is supported in $[-1,1]$, then
\begin{equation*}
\begin{split}
&  \sum\limits_{h_1,h_2,h_3,h_4\in \mathbb{Z}}  \prod\limits_{i=1}^4 W\left(\frac{h_i}{M}\right)\cdot \left(\frac{h_1h_2h_3+h_1h_2h_4+h_1h_3h_4+h_2h_3h_4}{r}\right)\\
\ll & \left(M^{1/2}r^{3/2}+M^2r^{1/2}+M^3\right)r^{\varepsilon}.
\end{split}
\end{equation*}
\end{Theorem}

Taking $M:=r^{1/2}$, we deduce the following.

\begin{Corollary} If $r$ is an odd prime and $W:\mathbb{R}\rightarrow \mathbb{C}$ is a Schwartz class function whose Fourier transform satisfies $0\le \hat{W}\le 1$ and is supported in $[-1,1]$, then
\begin{equation} \label{r1/2}
\sum\limits_{h_1,h_2,h_3,h_4\in \mathbb{Z}}  \prod\limits_{i=1}^4 W\left(\frac{h_i}{\sqrt{r}}\right)\cdot \left(\frac{h_1h_2h_3+h_1h_2h_4+h_1h_3h_4+h_2h_3h_4}{r}\right)\ll r^{7/4+\varepsilon}.
\end{equation}
\end{Corollary}

In \cite{PiXu}, Pierce and Xu established a remarkable bound for short character sums with general forms. However, for the special character sum with a Legendre symbol and a cubic form in four variables above, this general bound gives only $O(r^{2-1/56})$ for the left-hand side of \eqref{r1/2}, as compared to our bound $O(r^{2-1/4})$. 
We point out that the said form is an elementary symmetric polynomial. It would be interesting if one can establish bounds which improve Pierce's and Xu's result for other special forms and characters. 

\section{Possible ways to proceed}
In the bullet points below, we list possible ways to make further progess.
\begin{itemize} 
\item Try to generalize the established additive energy bounds for modular square roots in Theorems \ref{E2theo}and \ref{E4theo} from prime moduli to arbitrary moduli $r$. 
\item Try to improve these energy bounds. The ultimate goal is to make our improvements for the large sieve with square moduli unconditional. It needs to be said, though, that the methods in \cite{KSSZ}  are very elaborate, so this problem may be difficult. However, we don't need the full strength of our Hypothesis \ref{H2} to handle large moduli $r$ in our method non-trivially, as pointed out in section \ref{1}. So there is some hope.  
\item Prove an unconditional bound for the energy appearing in Hypothesis \ref{H3}.
\item The following question may be raised: If our additive energies are large, can we actually {\it utilize} this fact? Losely speaking, from additive combinatorics we know that large additive energies imply that there is a lot of additive structure.
\item Work out the connection between additive energies and special short character sums in section \ref{interesting} for general moduli (there we assumed $r$ to be prime). Do it also for higher additive energies such as $E_4(R;j,r)$.
\item This last point is vague but gives a possible direction: Exploit additive energies in connection to other problems related to the theory of exponential sums. 
\end{itemize}

\section{Appendix: Estimation of a complete exponential sum}
In this appendix we prove Lemma \ref{expsumbound}. Below we handle the case when $r$ is odd. Afterwards, we will indicate the modifications which need to be made in the case when $r$ is even. We begin by writing the exponential sum in question, defined in \eqref{complete}, in the form
\begin{equation} \label{completetrans} 
\begin{split}
\mathcal{E}_{j,h}(l,n)=& \sum\limits_{a \bmod{r}} \sum\limits_{\substack{k,\tilde{k} \bmod{r}\\ k^2\equiv ja\bmod{r}\\ \tilde{k}^2\equiv j(a+h)\bmod{r}}} e_r\left(l(\tilde{k}-k)+na\right)\\
=& \sum\limits_{\substack{k,\tilde{k}\bmod{r}\\ \tilde{k}^2-k^2\equiv jh\bmod{r}}} e_r\left(l(\tilde{k}-k)+n\overline{j}k^2 \right). 
\end{split}
\end{equation}
Now we make a change of variables 
\begin{equation*}
\begin{cases}
\tilde{k}-k\equiv b\bmod{r}\\
\tilde{k}+k\equiv c \bmod{r}.
\end{cases} 
\end{equation*}
Recalling that $r$ is odd, this is equivalent to
\begin{equation} \label{oddcase}
\begin{cases}
k\equiv \overline{2}(c-b)\bmod{r}\\
\tilde{k}\equiv \overline{2}(c+b) \bmod{r},
\end{cases} 
\end{equation}
where $\overline{2}\cdot 2\equiv 1\bmod{r}$.
It follows that
\begin{equation} \label{oddcaseE}
\begin{split}
\mathcal{E}_{j,h}(l,n)=& \sum\limits_{b \bmod{r}} \sum\limits_{\substack{c \bmod{r}\\ bc\equiv jh\bmod{r}}} e_r\left(lb+n\overline{4j}(c-b)^2\right),
\end{split}
\end{equation}
where $\overline{4j}\cdot 4j\equiv 1\bmod{r}$. 
Suppose that 
$$
d=(b,r).
$$
Then for the congruence 
$$
bc\equiv jh\bmod{r}
$$
to be solvable for $c$, it is necessary and sufficient that $d|h$. Set
$$
b_1:=\frac{b}{d}, \quad r_1:=\frac{r}{d}, \quad h_1:=\frac{h}{d}.
$$
Then it follows that $(b_1,r_1)=1$, and the said congruence is equivalent to
$$
c\equiv \overline{b_1}jh_1\bmod{r_1},
$$ 
where $\overline{b_1}b_1\equiv 1\bmod{r_1}$. 
Now dividing the $b$-sum in \eqref{completetrans} into subsums according to $d=(b,r)$, we obtain
\begin{equation*} 
\begin{split}
\mathcal{E}_{j,h}(l,n)=& \sum\limits_{d|(h,r)} \sum\limits_{\substack{b_1 \bmod{r_1}\\ (b_1,r_1)=1}} \sum\limits_{\substack{c \bmod{r}\\ c\equiv \overline{b_1}jh_1\bmod{r_1}}} e_r\left(ldb_1+n\overline{4j}(c-db_1)^2\right)\\
=&  \sum\limits_{d|(h,r)} \sum\limits_{\substack{b_1 \bmod{r_1}\\ (b_1,r_1)=1}} \sum\limits_{u\bmod{d}} e_r\left(ldb_1+n\overline{4j}(\overline{b_1}jh_1+ur_1-db_1)^2\right)\\
=& \sum\limits_{d|(h,r)} \sum\limits_{\substack{b_1 \bmod{r_1}\\ (b_1,r_1)=1}} e_{r_1}(lb_1) \sum\limits_{u\bmod{d}} e_r\left(n\overline{4j}(\overline{b_1}jh_1-db_1+ur_1)^2\right).
\end{split}
\end{equation*}
Here, we are free to fix $\overline{b_1}$ in such a way that $(\overline{b_1},d)=1$, which we want to assume in the following. Let 
$$
x:=n\overline{4j}, \quad y:=\overline{b_1}jh_1-db_1, \quad z:=\overline{b_1}jh_1. 
$$
Then the inner-most sum over $u$ turns into
\begin{equation*}
\begin{split}
\sum\limits_{u\bmod{d}} e_r\left(n\overline{4j}(\overline{b_1}jh_1-db_1+ur_1)^2\right)= &
\sum\limits_{u\bmod{d}} e_r\left(x(y+ur_1)^2\right)\\
= & e_r(xy^2)\sum\limits_{u\bmod{d}} e_d\left(x(r_1u^2+2yu)\right)\\
= & e_r\left(n\overline{4j}(\overline{b_1}jh_1-db_1)^2\right)\sum\limits_{u\bmod{d}} e_d\left(x(r_1u^2+2zu)\right).
\end{split}
\end{equation*}
Further, let 
$$
f=(n,d)
$$ 
and set 
\begin{equation*}
n_1:=\frac{n}{f}, \quad d_1:=\frac{d}{f}, \quad x_1:=\frac{x}{f}=n_1\overline{4j}.
\end{equation*}
Then 
\begin{equation*} 
\begin{split}
\sum\limits_{u\bmod{d}} e_d\left(x(r_1u^2+2zu)\right)=f\sum\limits_{v\bmod{d_1}} e_{d_1}\left(x_1(r_1v^2+2zv)\right).
\end{split}
\end{equation*}
Now we note that 
$$
(x_1,d_1)=(n_1,d_1)=1\quad \mbox{and}\quad (z,d_1)=(h_1,d_1).
$$ 
By Proposition \ref{Gausssums}, for the above quadratic Gauss sum over $v$ to be non-zero,  it is therefore necessary and sufficient that 
$$
(r_1,d_1)|h_1.
$$ 
In this case, we set 
$$
g:=(r_1,d_1), \quad r_2:=\frac{r_1}{g},\quad d_2:=\frac{d_1}{g}, \quad h_2:=\frac{h_1}{g}, \quad z_2:=\frac{z}{g}=\overline{b_1}jh_2.
$$
It then follows that
\begin{equation*} 
\sum\limits_{v\bmod{d_1}} e_{d_1}\left(x_1(r_1v^2+2zv\right)=g \sum\limits_{w\bmod{d_2}} e_{d_2}\left(x_1r_2w^2+2x_1z_2w\right)
\end{equation*}
with 
$$
(r_2,d_2)=1.
$$
Finally, applying Proposition \ref{Gausssums}, we get
\begin{equation}\label{thegauss} 
\begin{split}
\sum\limits_{w\bmod{d_2}} e_{d_2}\left(x_1(r_2w^2+2z_2w)\right)=&\epsilon_{d_2}\cdot e_{d_2}\left(-\overline{r_2}x_1z_2^2\right)\cdot \left(\frac{x_1r_2}{d_2}\right)\cdot \sqrt{d_2}\\ = & \epsilon_{d_2}\cdot e_{d_2}\left(-\overline{r_2}n_1j(\overline{2b_1}h_2)^2\right)\cdot \left(\frac{n_1jr_2}{d_2}\right)\cdot \sqrt{d_2},
\end{split}
\end{equation}
where $\overline{r_2}r_2\equiv 1\bmod{d_2}$. 
Combining everything above, we arrive at
\begin{equation} \label{newexp} 
\begin{split}
\mathcal{E}_{j,h}(l,n)= &\sum\limits_{d|(h,r)} fg \epsilon_{d_2}\cdot \left(\frac{n_1jr_2}{d_2}\right)\cdot \sqrt{d_2}\times\\ & \sum\limits_{\substack{b_1 \bmod{r_1}\\ (b_1,r_1)=1}} e_{r_1}(lb_1) e_r\left(n\overline{4j}(\overline{b_1}jh_1-db_1)^2\right)e_{d_2}\left(-\overline{r_2}n_1j(\overline{2b_1}h_2)^2\right),
\end{split}
\end{equation}
where we recall our conventions that $\overline{b_1}b_1\equiv 1\bmod{r_1}$ and $(\overline{b_1},d)=1$. Moreover,
\begin{equation*}
\begin{split}
e_r\left(n\overline{4j}(\overline{b_1}jh_1-db_1)^2\right)= &
e_r\left(fn_1\overline{4j}(\overline{b_1}jgh_2-fgd_2b_1)^2\right)\\
= & e_r\left(fg^2n_1\overline{4j}(\overline{b_1}jh_2-fd_2b_1)^2\right)\\
= & e_{r_2d_2}\left(n_1\overline{4j}(\overline{b_1}jh_2-fd_2b_1)^2\right)
\end{split}
\end{equation*}
since 
$$
\frac{r}{fg^2}=\frac{dr_1}{fg^2}=r_2d_2.
$$
It follows that
\begin{equation} \label{recip}
\begin{split}
& e_r\left(n\overline{4j}(\overline{b_1}jh_1-db_1)^2\right)e_{d_2}\left(-\overline{r_2}n_1j(\overline{2b_1}h_2)^2\right)
\\ = & e_{r_2d_2}\left(n_1\overline{4j}(\overline{b_1}jh_2-fd_2b_1)^2\right)e_{d_2}\left(-\overline{r_2}n_1\overline{4j}(\overline{b_1}jh_2-fd_2b_1)^2\right)\\
= & e_{r_2}\left(\overline{d_2}n_1\overline{4j}(\overline{b_1}jh_2-fd_2b_1)^2\right),
\end{split}
\end{equation}
where $\overline{d_2}d_2\equiv 1\bmod{r_2}$. Here we have used
the reciprocity law 
$$
\frac{1}{r_2d_2}\equiv \frac{\overline{r_2}}{d_2}+\frac{\overline{d_2}}{r_2}\bmod{1}
$$
for Kloosterman fractions in the second equation in \eqref{recip}.
Combining \eqref{newexp} and \eqref{recip}, we obtain
\begin{equation} \label{expsuminterest}  
\begin{split}
\mathcal{E}_{j,h}(l,n)= & \sum\limits_{d|(h,r)} fg \epsilon_{d_2}\cdot \left(\frac{n_1jr_2}{d_2}\right)\cdot \sqrt{d_2}\cdot \sum\limits_{\substack{b_1 \bmod{r_1}\\ (b_1,r_1)=1}} e_{r_1}(lb_1)e_{r_2}\left(\overline{d_2}n_1\overline{4j}(\overline{b_1}jh_2-fd_2b_1)^2\right)\\
= &  \sum\limits_{d|(h,r)} fg \epsilon_{d_2}\cdot \left(\frac{n_1jr_2}{d_2}\right)\cdot \sqrt{d_2}\cdot \sum\limits_{\substack{b_1 \bmod{r_1}\\ (b_1,r_1)=1}} e_{r_1}\left(lb_1+n_1g\overline{d_2}\cdot\frac{(jh_2-fd_2b_1)^2}{4jb_1^2}\right),
\end{split}
\end{equation}
where the denominator $4jc^2$ indicates a multiplicative inverse modulo $r_1$. Now we define
$$
\mathcal{G}(q;a,b,j,k,u,s):=\sum\limits_{\substack{c=1\\ (c,q)=1}}^{q}e_{q}\left(ac+b\cdot \frac{(jk-us^2c^2)^2}{4js^3c^2}\right),
$$ 
where we assume that $(js,q)=1$. Then the exponential sum in the last line of \eqref{expsuminterest} equals
\begin{equation} \label{expequals}
\sum\limits_{\substack{b_1 \bmod{r_1}\\ (b_1,r_1)=1}} e_{r_1}\left(lb_1+n_1\cdot\frac{(jh_2-fd_2b_1)^2}{4jb_1^2}\right)=
\mathcal{G}(r_1;l,n_1g\overline{d_2},j,h_2,fd_2,1).
\end{equation}
The exponential sum $\mathcal{G}(q;a,b,j,k,u,s)$ has the following multiplicative property.

\begin{Lemma} Suppose that $q_1,q_2\in \mathbb{N}$ with $(q_1,q_2)=1$ and $a,b,j,k,u,s\in \mathbb{Z}$, where $(js,q)=1$. Then 
\begin{equation} \label{multi}
\mathcal{G}(q_1q_2;a,b,j,k,u,s)=\mathcal{G}(q_1;a,b,j,k,u,sq_2)\mathcal{G}(q_2;a,b,j,k,u,sq_1).
\end{equation}
\end{Lemma}

\begin{proof}
Multiplying out the square in the numerator, we get
\begin{equation} \label{used}
\mathcal{G}(q;a,b,j,k,u,s):=\sum\limits_{\substack{c=1\\ (c,q)=1}}^{q}e_{q}\left(ac+b\left(\frac{jk^2}{4s^3c^2}-\frac{ku}{2s}+\frac{u^2sc^2}{4j}\right)\right).
\end{equation}
Hence, the right-hand side of \eqref{multi} becomes
\begin{equation*}
\begin{split}
&\mathcal{G}(q_1;a,b,j,k,u,sq_2)\mathcal{G}(q_2;a,b,j,k,u,sq_1)\\ = &  
\left(\sum\limits_{\substack{c_1=1\\ (c_1,q_1)=1}}^{q_1}e_{q_1}\left(ac_1+\overline{4}b\left(jk^2\overline{s^3q_2^3c_1^2}-2ku\overline{sq_2}+\overline{j}u^2sq_2c_1^2\right)\right)\right)\times\\ & \left(\sum\limits_{\substack{c_2=1\\ (c_2,q_2)=1}}^{q_2}e_{q_2}\left(ac_2+\overline{4}b\left(jk^2\overline{s^3q_1^3c_2^2}-2ku\overline{sq_1}+\overline{j}u^2sq_1c_2^2\right)\right)\right)
\\
= 
& \sum\limits_{\substack{c_1=1\\ (c_1,q_1)=1}}^{q_1}\sum\limits_{\substack{c_2=1\\ (c_2,q_2)=1}}^{q_2}
e_q\left(a\left(c_1q_2+c_2q_1\right)+\overline{4}b\left(jk^2\overline{s}^3\left(\overline{c_1}^2\overline{q_2}^2+\overline{c_2}^2\overline{q_1}^2\right)-2ku\overline{s}\left(\overline{q_1}q_1+\overline{q_2}q_2\right)+\overline{j}u^2s\left(c_1^2q_2^2+c_2^2q_1^2\right)\right)\right).
\end{split}
\end{equation*}
As $c_i$ runs over all reduced residue classes modulo $q_i$ for $i=1,2$, $c_1q_2+c_2q_1$ runs over all reduced residue classes modulo $q$. Moreover, we check that 
\begin{equation*}
\begin{split}
(c_1q_2+c_2q_1)^2\equiv & c_1^2q_2^2+c_2^2q_1^2\bmod{q},\\
\overline{q_1}q_1+\overline{q_2}q_2 \equiv & 1 \bmod{q},\\
\overline{c_1q_2+c_2q_1}^2\equiv & \overline{c_1}^2\overline{q_2}^2+\overline{c_2}^2\overline{q_1}^2 \bmod{q}
\end{split}
\end{equation*}
using the Chinese remainder theorem. It follows that
\begin{equation*}
\begin{split}
\mathcal{G}(q_1;a,b,j,k,u,sq_2)\mathcal{G}(q_2;a,b,j,k,u,sq_1)
= &
\sum\limits_{\substack{c=1\\ (c,q)=1}}^{q}
e_q\left(ac+\overline{4}b(jk^2\overline{s}^3\overline{c}^2-2ku\overline{s}+\overline{j}u^2sc^2)\right)\\
=& \mathcal{G}(q;a,b,j,k,u,s),
\end{split}
\end{equation*}
where for the second equation, we have used \eqref{used} again.  
\end{proof}

So to estimate $\mathcal{G}(q;a,b,j,k,u,s)$, it suffices to bound $\mathcal{G}(p^m;a,b,j,k,u,s)$
for prime powers $q=p^m$ with $p\not=2$ (since we assumed $r$ to be odd). 
Our goal is to prove that
\begin{equation} \label{estigoal} 
\mathcal{G}(q;a,b,j,k,u,s)\le 12q^{4/5}(bu^2,a,bk^2,q)^{1/5} \quad \mbox{ if } q=p^m.
\end{equation}
We write 
\begin{equation} \label{place}
\mathcal{G}(p^m;a,b,j,k,u,s)=\sum\limits_{\substack{x=1\\ (x,p)=1}}^{p^m} e(f(x)),
\end{equation}
where
$$
f(x)=ax+b\cdot \frac{(jk-us^2x^2)^2}{4js^3x^2}=ax+b\left(\frac{jk^2}{4s^3x^2}-\frac{ku}{2s}+\frac{u^2sx^2}{4j}\right).
$$
We calculate that
$$
f'(x)=a+b\left(-\frac{jk^2}{2s^3x^3}+\frac{u^2sx}{2j}\right)=\frac{bu^2s^4x^4+2ajs^3x^3-bj^2k^2}{2js^3x^3}.
$$
Since $(2js,p)=1$, we have $\mbox{ord}_p(f')=t$, where $p^t=(bu^2,a,bk^2,p^m)$. If $t=m$, then \eqref{estigoal} holds trivially. If $t=m-1$, then dividing $f(x)$ by $p^t$, we obtain an exponential sum modulo $p$ which can be estimated by Proposition \ref{primemoduli}, thus establishing \eqref{estigoal} as well. If $p\not= 2$ and $m\ge t+2$, then Proposition \ref{primepowermoduli} implies \eqref{estigoal} since the number of roots of $p^{-t}f'(x)$ modulo $p$ is bounded by $4$ and their multiplicities as well. Using \eqref{multi} and \eqref{estigoal}, we deduce that
$$
\mathcal{G}(q;a,b,j,k,u,s)\ll_{\varepsilon} q^{4/5+\varepsilon}(bu^2,a,bk^2,q)^{1/5} \quad \mbox{ for all odd } q\in \mathbb{N}. 
$$
It follows that 
\begin{equation} \label{necest}
\mathcal{G}(r_1;l,n_1,j,h_2,fd_2,1)\ll r_1^{4/5+\varepsilon}(l,r_1)^{1/5}.
\end{equation}
Combining \eqref{expsuminterest}, \eqref{expequals} and \eqref{necest}, we obtain
\begin{equation} \label{expsuminterestbound}
\begin{split}
|\mathcal{E}_{j,h}(l,n)|\le & \sum\limits_{d|(h,r)} fgd_2^{1/2}r_1^{4/5+\varepsilon}(l,r_1)^{1/5}\\
= & \sum\limits_{d|(h,r)} dd_2^{-1/2} r_1^{4/5+\varepsilon}(l,r_1)^{1/5} \\
\ll & r^{4/5+2\varepsilon} (h,r)(l,r)^{1/5}.
\end{split}
\end{equation}
This completes the proof of Lemma \ref{expsumbound} if $r$ is odd.

If $r$ is even, then we follow the same procedure with some small modifications. Firstly, if $r$ is even, then the system of congruences in \eqref{oddcase} needs to be replaced by 
\begin{equation*}
\begin{cases}
k\equiv (c-b)/2\bmod{r/2}\\
\tilde{k}\equiv (c+b)/2 \bmod{r/2}.
\end{cases} 
\end{equation*}
The equation \eqref{oddcaseE} then needs to be modified accordingly. Secondly, the evaluation of the quadratic Gauss sum in \eqref{thegauss} slightly changes if $d_2$ is even. Thirdly, $\mathcal{G}(r_1;l,n_1g\overline{d_2},j,h_2,fd_2,1)$ takes a slightly different form if $r_1$ is even, and we also need to bound the relevant exponential sums modulo powers of 2. This is not a problem since Proposition \ref{primepowermoduli} covers the case of powers of 2. Altogether, we arrive at the same estimate, proving Lemma \ref{expsumbound} for the case when $r$ is even as well.

\end{document}